\documentclass[11pt]{article}

\usepackage[english]{babel}
\usepackage{upref,amsfonts,amsxtra,latexsym,color}
\usepackage{amssymb,amsmath,amsthm,a4wide,epsf,mathrsfs,verbatim,hyperref}
\usepackage[all]{xy}
 
\newtheorem{theorem}{Theorem}[subsection]
\newtheorem{lemma}[theorem]{Lemma}
\newtheorem{corollary}[theorem]{Corollary}
\newtheorem{proposition}[theorem]{Proposition}

\theoremstyle{definition}
\newtheorem{definition}[theorem]{Definition}
\newtheorem{remark}[theorem]{Remark}
\newtheorem{example}[theorem]{Example}
\newtheorem{question}[theorem]{Question}

\numberwithin{equation}{section}
\numberwithin{theorem}{section}

\newcommand{\cM}{{\cal M}}
\newcommand{\ep}{{\varepsilon}}
\newcommand{\cN}{{\cal N}} 
\newcommand{\cA}{{\cal A}}

\newcommand{\cU}{{\cal U}}
\newcommand{\cP}{{\cal P}}

\newcommand{\cB}{{\cal B}}
\newcommand{\cD}{{\cal D}}
\newcommand{\cE}{{\cal E}}
\newcommand{\cO}{{\cal O}}

\newcommand{\cZ}{{\mathcal Z}}

\newcommand{\C}{{\mathbb C}}
\newcommand{\Z}{{\mathbb Z}}
\newcommand{\N}{{\mathbb N}}

\newcommand{\R}{{\mathbb R}}
\newcommand{\Cs}{{$C^*$-al\-ge\-bra}}

\newcommand{\sh}{{$^*$-ho\-mo\-mor\-phism}}

\date{}
\title{Irreducible inclusions of simple \Cs s}

\author{Mikael R\o rdam}

\begin{document}

\maketitle

\begin{center}
\emph{Dedicated to the memory of Vaughan Jones}
\end{center}

\begin{abstract} 
\noindent The literature contains interesting examples of inclusions of simple \Cs s with the property that all intermediate \Cs s likewise are simple. In this article we take up a systematic study of such inclusions, which we refer to as being \emph{$C^*$-irreducible} by the analogy that all intermediate von Neumann algebras of an inclusion of factors are again factors precisely when the given inclusion is irreducible. 

We give an intrinsic characterization of when an inclusion of \Cs s is $C^*$-irreducible, and use this to revisit known and exhibit new $C^*$-irreducible inclusions arising from groups and dynamical systems. Using a theorem of Popa one can show that an inclusion of II$_1$-factors is $C^*$-irreducible if and only if it is irreducible with finite Jones index. We further show how one can construct $C^*$-irreducible inclusions from inductive limits, and we 
discuss how the notion of $C^*$-irreducibility behaves under tensor products.
\end{abstract}

\section{Introduction}

\noindent Jones' index of subfactors, \cite{Jones:index}, and the subsequent classification of subfactors of finite depth, are famous examples of the rich mathematical structure possessed by  inclusions of operator algebras. The classification of hyperfinite von Neumann factors was an inspiration for the recently (almost) completed classification of simple nuclear \Cs s (the Elliott program), and the Jones' theory of subfactors has likewise been a model for understanding inclusions of \Cs s, see for example \cite{Izumi:Crelle2002}.

Inclusions of \Cs s are ubiquitous in operator algebras, and are perhaps most often encountered in $C^*$-dynamics. If a group $\Gamma$ acts on a \Cs{} $\cA$, then both $\cA$ and the  group \Cs{} of $\Gamma$ are sub-algebras of the crossed product of $\cA$ with $\Gamma$; and the fixed-point algebra $\cA^\Gamma$ is a subalgebra of $\cA$. An inclusion of groups $\Lambda \subseteq \Gamma$ gives rise to an inclusion of group \Cs s and von Neumann algebras, and, whenever $\Gamma$ acts on a given \Cs, also of their crossed products. 

Given an inclusion $\cB \subseteq \cA$ of \Cs s, it is natural to consider the lattice of intermediate \Cs s. In some special cases of interest one can even classify these, or one can show that each intermediate algebra shares properties with $\cA$ and $\cB$. Perhaps the best understood inclusions are those arising from crossed products:  $\cB \subseteq \cB \rtimes_{\mathrm{red}} \Gamma$, where  $\cB$ is simple and the action $\Gamma \curvearrowright \cB$ is outer.  It is an easy consequence of classical results of Kishimoto, Olesen--Pedersen from ca.\ 1980 that any intermediate \Cs{} of these  inclusions are simple, making the inclusions $C^*$-simple,  see Section~\ref{sec:dynamics}. Moreover, there is a Galois type correspondance between intermediate \Cs s  of such an inclusion and subgroups of $\Gamma$ as shown by Izumi, \cite{Izumi:Crelle2002}, and
 Cameron an Smith,  \cite{CamSmith:Galois}. A similar Galois type correspondance was established in \cite{BEFPR-PAMS21} of intermediate \Cs s of the inclusion of a groupoid \Cs{} and its canonical Cartan subalgebra in terms of sub-groupoids.

Amrutam and Kalantar show in \cite{Amrutam-Kalantar-ETDS2020} that the inclusion $C^*_\lambda(\Gamma) \subseteq \cA \rtimes_{\mathrm{red}} \Gamma$ is $C^*$-irreducible under certain ``mixing'' conditions on the action, and Amrutam and Ursu proved a Galois correspondance for the intermediate \Cs s of these inclusions under some further conditions.  In Theorem~\ref{thm:AK} we improve the first mentioned result, giving a necessary and sufficient condition for the inclusion $C^*_\lambda(\Gamma) \subseteq \cA \rtimes_{\mathrm{red}} \Gamma$ to be $C^*$-irreducible.

 Izumi--Longo--Popa established in \cite{ILP:JFA1998} a Galois correspondance between intermediate von Neumann algebras of the inclusion $\mathcal{M}^G \subseteq \mathcal{M}$  and subgroups of $G$,
where $\mathcal{M}$ is a factor and $G$ is a compact group equipped with a minimal action on $\mathcal{M}$. Izumi established in \cite{Izumi:Crelle2002} a similar Galois correspondance for the inclusion $\cA^\Gamma \subseteq \cA$ arising from an outer action of a finite group $\Gamma$ on a unital simple \Cs{} $\cA$. In particular, the inclusion $\cA^\Gamma \subseteq \cA$ is $C^*$-irreducible.  

Another Galois correspondance was established by Suzuki, \cite{Suzuki:intermediate}, who proved that any intermediate \Cs{} of an inclusion $C(Y) \rtimes \Gamma \subseteq C(X) \rtimes \Gamma$ is of the form $C(Z) \rtimes \Gamma$ (under suitable conditions explained in Section~\ref{sec:background} below). 

The purpose of this paper is to explain when an inclusion of \Cs s is $C^*$-irreducible, i.e.,  has the property that all intermediate \Cs s are simple, and to provide (new) examples of such inclusions. We expect that the lattice of intermediate \Cs s of such inclusions will exhibit a more rigid structure allowing for a better understanding; and, indeed as described above, complete classifications of the lattice of intermediate \Cs s of $C^*$-irreducible inclusions have been obtained in several cases of interest. In 
Chapter~\ref{sec:irr-inclusion} we give an intrinsic characterization of when an inclusion of simple \Cs s is $C^*$-irreducible. The condition states that an inclusion $\cB \subseteq \cA$ of unital \Cs s is $C^*$-irreducible if and only if each non-zero positive element in $\cA$ is full relatively to $\cB$ 
(as defined in Chapter~\ref{sec:irr-inclusion}). This condition is clearly sufficient for $C^*$-irreducibility, and we prove it is also necessary.

Being $C^*$-irreducible is a stronger property than the usual notion of irreducibility of an inclusion, namely that the relative commutant is trivial. In Section~\ref{sec:vNalg} we consider when an inclusion of von Neumann factors is $C^*$-irreducible. By a theorem of Popa, this happens for an inclusion of II$_1$-factors if and only if the inclusion is irreducible (in the usual sense) and of finite Jones' index. In Section~\ref{sec:dynamics} we consider when inclusions arising from groups and dynamical systems are $C^*$-irreducible, and we revisit here some of the results mentioned above.

We show in Section~\ref{sec:indlimit} how to construct $C^*$-irreducible inclusions from inductive limits. In particular we give examples of $C^*$-irreducible inclusions of UHF-algebras. The construction involves the relative position of finite dimensional sub-\Cs s inside another finite dimensional \Cs, and leads to the notion that we call \emph{everywhere non-orthogonal} subalgebras. 

Finally, in Section~\ref{sec:tensor}, we show how $C^*$-irreducible inclusions behave under taking tensor products. We begin in Section~\ref{sec:background} by explaining in more detail some key examples that motivated this paper.

I thank  Pierre de la Harpe, Adrian Ioana, Mehrdad Kalantar, Tron Omland, Sorin Popa and Yuhei Suzuki for their very helpful input to this paper.

\section{Inclusions of \Cs s and their intermediate \Cs s} \label{sec:background}

\noindent This section contains a review of some results on inclusions of \Cs s that prompted me to formalize the idea of a $C^*$-irreducible inclusion and to write this paper. We focus in particular on results that provide an analog of the Galois correspondance stating that all intermediate \Cs s of a given inclusion share common properties or arise in a certain way. One famous example of this situation is found in the  tensor splitting theorem below by Ge and Kadison on intermediate von Neumann factors of an inclusion of tensor products, as well as a $C^*$-analog of this result.

\begin{theorem}[Ge--Kadison, \cite{Ge-Kadison;Invent1996}] Let $\cM$ be a von Neumann factor, let $\cN$ be an arbitrary von Neumann algebra, and let $\mathcal{P}$ be a von Neumann algebra such that $\cM \overline{\otimes} \C \subseteq \mathcal{P} \subseteq \cM \overline{\otimes} \cN$. Then $\mathcal{P} = \cM \overline{\otimes} \mathcal{N}_0$ for some von Neumann subalgebra $\mathcal{N}_0$ of $\cN$. 
\end{theorem}

\noindent In the theorem above $\overline{\otimes}$ denotes the (spatial) von Neumann algebra tensor product. We let $\otimes$ denote the minimal $C^*$-tensor product.

Let $\cA$ and $\cB$ be a unital \Cs s. For each $\varphi \in \cA^*$  there exists a bounded linear map $R_\varphi \colon \cA \otimes \cB \to \cB$, called  a \emph{slice map}, satisfying $R_\varphi(a \otimes b) = \varphi(a) b$, for $a \in \cA$ and $b \in \cB$. The \Cs{} $\cA$ is said to have \emph{Wassermann's property (S)}  if for each unital inclusion $\cB_0 \subseteq \cB$ of \Cs s and for each $x \in \cA \otimes \cB$, one has $x \in \cA \otimes \cB_0$ if and only if $R_\varphi(x) \in \cB_0$, for all $\varphi \in \cA^*$, cf.\ \cite{Wass:FubiniI}. Wassermann proved in \cite{Wass:FubiniII} that all nuclear \Cs s have property (S). 

The result below was obtained independently by Zsido, \cite{Zsido:PAMS1999}, and Zacharias, \cite{Zach:PAMS2001}.

\begin{theorem}[Zacharias, Zsido] \label{thm:tensorI} Let $\cE$ be a unital \Cs. Then $\cE$ is simple and has property (S) of Wassermann if and only if for each unital \Cs{} $\cB$ and each intermediate \Cs{} $\cE \otimes \C \subseteq \cD \subseteq \cE \otimes \cB$, one has $\cD = \cE \otimes \cB_0$ for some sub-\Cs{} $\cB_0$ of $\cB$.
\end{theorem}

\noindent This theorem immediately implies simplicity of any intermediate \Cs{} of an inclusion $\mathcal{E} \otimes \cB \subseteq \mathcal{E} \otimes \cA$, where $\mathcal{E}$ is simple and has property (S) and $\cB \subseteq \cA$ is $C^*$-irreducible, cf.\ Theorem~\ref{thm:tensorII}. Indeed, any such intermediate \Cs{} will be of the form $\mathcal{E} \otimes \mathcal{D}$ for some \Cs{} $\cB \subseteq \mathcal{D} \subseteq \cA$. Hence $\mathcal{D}$ is simple, so 
$\mathcal{E} \otimes \mathcal{D}$ is also simple by Takesaki's theorem. 

Describing all intermediate \Cs s of inclusion of the more general form $\cB_1 \otimes \cB_2 \subseteq \cA_1 \otimes \cA_2$, when $\cB_j \subseteq \cA_j$ are $C^*$-irreducible, $j=1,2$, is more tricky. We shall give partial results about such inclusions in Section~\ref{sec:tensor}.

We now turn to inclusions arising from dynamical systems. We already mentioned in the introduction inclusions of \Cs s arising from crossed products, and we shall get back to those examples in Section~\ref{sec:dynamics}. Suzuki, \cite{Suzuki:intermediate},
established  another Galois type correspondance between intermediate locally compact Hausdorff spaces equipped with an action of a fixed group $\Gamma$, and of their associated crossed product \Cs s:

\begin{theorem}[{Suzuki, \cite[Theorem 2.3]{Suzuki:intermediate}}] \label{thm:Suzuki}
Let $\Gamma$ be a discrete group with the approximation property (AP) of Haagerup and Kraus. Let $X$ and $Y$ be locally compact Hausdorff spaces with  free $\Gamma$-actions for which there is a surjective continuous $\Gamma$-equivariant map $Y \to X$.

It follows that the map $Z \mapsto C_0(Z) \rtimes_\mathrm{red} \Gamma$ is a bijection from the lattice of  $\Gamma$-spaces $Z$ between $X$ and $Y$ to the lattice of intermediate \Cs s of the inclusion $C_0(X) \rtimes_\mathrm{red} \Gamma \subseteq C_0(Y) \rtimes_\mathrm{red}  \Gamma$.
\end{theorem}

\noindent In other words, each intermediate \Cs{}  $C_0(X) \rtimes_\mathrm{red} \Gamma \subseteq \cD \subseteq C_0(Y) \rtimes_\mathrm{red}  \Gamma$ is of the form $\cD = C_0(Z) \rtimes_\mathrm{red} \Gamma$ for some
$\Gamma$-equivariant intermediate \Cs{} $C_0(X) \subseteq C_0(Z) \subseteq C_0(Y)$, under the assumption of Theorem~\ref{thm:Suzuki}.

We mention also the following result on the existence of ``tight inclusions'' from \cite{Suzuki:intermediate} establishing situations where the lattice of intermediate \Cs s is trivial for non-trivial reasons:

\begin{theorem}[{Suzuki, \cite[Theorem 5.1]{Suzuki:intermediate}}] Let $\cA$ be a simple \Cs{} that tensorially absorbs the Cuntz algebra $\cO_\infty$. Then $\cA$ admits an endomorphism $\sigma$ for which the inclusion $\sigma(\cA) \subseteq \cA$ on the one hand is non-trivial in the sense that there is no conditional expectation $\cA \to \sigma(\cA)$, while on the other hand it has no proper intermediate \Cs s, i.e., the inclusion is \emph{tight}. 
\end{theorem}

\noindent Suzuki has more results about tight inclusions in \cite{Suzuki:tight}. 

A discrete group $\Gamma$ is said to be \emph{$C^*$-simple} if its reduced group \Cs{} $C^*_\lambda(\Gamma)$ is simple. Amrutam and Kalantar prove in \cite[Theorem 1.1]{Amrutam-Kalantar-ETDS2020} that for certain actions of a $C^*$-simple group $\Gamma$ on a unital \Cs{} $\cA$, each intermediate \Cs{} $\mathcal{D}$ of the inclusion $C^*_\lambda(\Gamma) \subseteq \cA \rtimes_\mathrm{red} \Gamma$ is simple.   
As a corollary they obtain that for each minimal action of a $C^*$-simple group $\Gamma$ on a  compact Hausdorff space $X$, each intermediate \Cs{} $C^*_\lambda(\Gamma) \subseteq \cD  \subseteq C(X) \rtimes_\mathrm{red} \Gamma$ is simple. In Theorem~\ref{thm:AK} we extend these results and provide a (partly) new proof based on the techniques developed here.

Amrutam and Ursu, \cite{Amrutam-Ursu-2021}, consider the more general situation of a group $\Gamma$ acting minimally on compact Hausdorff spaces $X$ and $Y$, where $Y$ is a $\Gamma$-invariant factor of $X$, ensuring that we have a $\Gamma$-equivariant inclusion $C(Y) \subseteq C(X)$. They show that if $C(Y) \rtimes_\mathrm{red} \Gamma$ is simple, then so is each intermediate \Cs{} of the inclusion $C(Y) \rtimes_\mathrm{red} \Gamma  \subseteq C(X) \rtimes_\mathrm{red} \Gamma$,
cf.\ \cite[Theorem 1.5]{Amrutam-Ursu-2021}. If one further assumes that 
$\Gamma$ has property (AP) and the action is free, then Suzuki's Theorem~\ref{thm:Suzuki}  implies that any such intermediate \Cs{}  is of the form $C(Z) \rtimes_\mathrm{red} \Gamma$, which again yields simplicity.

The strategy of proof used in both the Amrutam--Kalantar and the Amrutam--Ursu papers rely on clever generalizations of the \emph{Powers' averaging property}, which in its original form states that
$$\tau_0(x) \cdot 1 \in \overline{\mathrm{conv}\{u_s^* x u_s : s \in \Gamma\}},$$
for each $x \in C^*_\lambda(\Gamma)$, when $\Gamma$ is a free group (on two or more generators), where $\tau_0$ is the canonical trace on $C^*_\lambda(\Gamma)$, and where $\{u_t\}_{t \in \Gamma}$ is the unitary representation of $\Gamma$ in $ C^*_\lambda(\Gamma)$. This, in turn, implies that $C^*_\lambda(\Gamma)$ is simple with unique trace. It was later shown independently by Haagerup, \cite{Haagerup:new-look}, and Kennedy, \cite{Kennedy:C*-simple}, that the Powers' averaging property above, in fact, characterizes $C^*$-simple groups $\Gamma$.

\section{Irreducible inclusions of \Cs s} \label{sec:irr-inclusion}

\noindent Here is the main definition of this paper:

\begin{definition} \label{def:C*-irr}
A unital inclusion $\cB \subseteq \cA$ of simple unital \Cs s is said to be \emph{$C^*$-irreducible} if each intermediate \Cs{} $\cB \subseteq \cD \subseteq \cA$ is simple. (The inclusion $\cB \subseteq \cA$ is \emph{unital} if the unit of $\cA$ belongs to $\cB$.)
\end{definition}

\noindent  It will be shown in Remark~\ref{rem:irr} below that any $C^*$-irreducible inclusion $\cB \subseteq \cA$ is irreducible (in the sense that $\cB' \cap \cA = \C$), while the converse is not true in general, even under the additional assumption that $\cA$ and $\cB$ are simple. In some instances, irreducibility and $C^*$-irreducibility do agree, see for example Theorem~\ref{thm:crossed-product}, and see also Theorem~\ref{thm:vNirr}.

Further justification of the terminology  of Definition~\ref{def:C*-irr} will be given in the next section. We proceed to give an intrinsic description of $C^*$-irreducible inclusions. First we need the following:

\begin{lemma} \label{lm:span}
Let $\cA$ be a unital \Cs{} and let $\mathcal{W}$ be a subset of $\cA$. Let $a \in \cA^+$, and suppose that there exist $x_1, \dots, x_n \in \mathrm{span}(\mathcal{W})$ such that $\sum_{j=1}^n x_j^*ax_j \ge 1_\cA$. Then there exist $w_1, \dots, w_m \in \mathcal{W}$ such that $\sum_{j=1}^m w_j^*aw_j \ge 1_\cA$.
\end{lemma}

\begin{proof} It suffices to show that for each $x \in \mathrm{span}(\mathcal{W})$ there exist $w_1, \dots, w_k 
\in \mathcal{W}$ such that $\sum_{j=1}^k w_j^*aw_j \ge x^*ax$.  Write $x = \sum_{j=1}^\ell \lambda_j w_j$, with $w_j \in \mathcal{W}$ and $\lambda_j \in \C$. Since $v^*aw+w^*av\le v^*av+w^*aw$, for all $v,w \in \cA$, we get
$$x^*ax = \sum_{i,j=1}^\ell (\lambda_j w_j)^*a(\lambda_iw_i) \le \ell \sum_{j=1}^\ell |\lambda_j|^2 w_j^*aw_j.$$
Upon repeating each $w_j$ at least $\ell \cdot |\lambda_j|^2$ times, and after relabelling the $w_j$'s, we obtain $x^*ax \le \sum_{j=1}^k w_j^*aw_j$, as desired.
\end{proof}

\noindent An element $a$ in a \Cs{} $\cA$ is \emph{full} if it is not contained in any proper closed two-sided ideal of $\cA$. The following lemma is elementary and well-known, see eg.\ 
\cite[Exercise 4.8]{RorLarLau:k-theory} for (i) $\Rightarrow$ (ii), and use Lemma~\ref{lm:span} to see that (ii) $\Rightarrow$ (iii):

\begin{lemma} \label{lm:full}
The following conditions are equivalent for each positive element $a$ in a unital \Cs{} $\cA$:
\begin{enumerate}
\item $a$ is full in $\cA$,
\item there exist $x_1, \dots, x_n \in \cA$ such that $\sum_{j=1}^n x_j^*bx_j \ge 1_\cA$,
\item there exist unitaries $u_1, \dots, u_m \in \cA$ such that $\sum_{j=1}^m u_j^*bu_j \ge 1_\cA$.
\end{enumerate}
\end{lemma}

\begin{definition} Let $\cB \subseteq \cA$ be a unital inclusion of \Cs s. A positive element $a \in \cA$ is said to be \emph{full relatively to $\cB$} if there exist elements $x_1, \dots, x_n \in \cB$ such that $\sum_{j=1}^n x_j^*ax_j \ge 1_\cA$. 
\end{definition}

\noindent The property of being relatively full as defined above can be reformulated as follows:

\begin{lemma} \label{lm:rel-full}
Let $\cB \subseteq \cA$ be a unital inclusion of \Cs s. Let $\mathcal{W}$ be a subset of $\cB$ such that $\mathrm{span}(\mathcal{W})$ is dense in $\cB$. Then the following conditions are equivalent for each positive element  $a \in \cA$:
\begin{enumerate}
\item $a$ is full relatively to $\cB$,
\item  there exist $x_1, \dots, x_n \in \cB$ such that $\sum_{j=1}^n x_j^* a x_j$ is invertible in $\cA$,
\item   there exist $w_1, \dots, w_m \in \mathcal{W}$ such that $\sum_{j=1}^m w_j^* a w_j \ge 1_\cA$,
\item  there exist unitaries $u_1, \dots, u_m \in \cB$ such that $\sum_{j=1}^m u_j^* a u_j \ge 1_\cA$.
\end{enumerate}
\end{lemma}

\begin{proof} (iv) $\Rightarrow$ (i) $\Rightarrow$ (ii) are trivial, and (iii) $\Rightarrow$ (iv) follows from the fact that any unital \Cs{} is the span of its unitary elements. 

(ii) $\Rightarrow$ (iii). Note that the sum $\sum_{j=1}^n x_j^* a x_j$ is invertible if and only if $\sum_{j=1}^n x_j^* a x_j \ge \delta \cdot 1_\cA$, for some $\delta >0$. Approximate each $x_j$ by elements $y_j \in \mathrm{span}(\mathcal{W})$ close enough so that $\sum_{j=1}^n y_j^* a y_j \ge (\delta/2) \cdot 1_\cA$. Multiplying each $y_j$ by $(\delta/2)^{-1/2}$ we obtain that $\sum_{j=1}^n y_j^* a y_j \ge 1_\cA$. It now follows from Lemma~\ref{lm:span} that (iii) holds for suitable $w_1, \dots, w_m \in \mathcal{W}$. 
\end{proof}

\noindent The next result connects the notion of relative fullness to the usual notion of fullness:

\begin{proposition} \label{prop:rel-full}
Let $\cB \subseteq \cA$ be a unital inclusion of \Cs s. A positive element $a \in \cA$ is full relatively to $\cB$ if and only if $a$ is full (in the usual sense) in the \Cs{} $C^*(\cB,a)$ generated by $\cB$ and $a$.
\end{proposition}

\begin{proof} If $a$ is full relatively to $\cB$, then $a$ is full in $C^*(\cB,a)$ (eg., by Lemma~\ref{lm:full}).

Suppose conversely that $a$ is full in $C^*(\cB,a)$. Let $\mathcal{W}$ be the set of elements in $C^*(\cB,a)$ of the form $w = b_1ab_2 \cdots b_{r-1}ab_r$, with $r \ge 1$ and $b_j \in \cB$. As we may take $b_1$ and/or $b_r$ to be $1_\cB$, we see that $\mathrm{span}(\mathcal{W})$ is dense in $C^*(\cB,a)$. 

By Lemma~\ref{lm:rel-full} there exist $w_1, \dots, w_m \in \mathcal{W}$ such that $\sum_{j=1}^m w_j^*aw_j \ge 1_\cA$. To complete the proof we show that whenever $w = b_1ab_2 \cdots b_{r-1}ab_r \in \mathcal{W}$, with $b_j \in \cB$, then there exists $x \in \cB$ such that $x^*ax \ge w^*aw$. There is nothing to prove if $r =1$. For $r \ge 2$ we may write $w = v a b$, where $v \in \mathcal{W}$ and $b \in \cB$. Now,
$$w^*aw = b^*av^*avab = b^*a^{1/2}\big(a^{1/2}v^*ava^{1/2}\big)a^{1/2}b \le \|a^{1/2}v^*ava^{1/2}\| \, b^*ab = x^*ax,$$
when $x = \|a^{1/2}v^*ava^{1/2}\|^{1/2} b \in \cB$. 
\end{proof}

\noindent
 For the next result, that characterizes $C^*$-irreducible inclusions, recall that a \Cs{} is said to have property (SP) (for small projections) if each non-zero hereditary sub-\Cs{} contains a non-zero projection. Clearly all \Cs s of real rank zero (and hence all von Neumann algebras) have property (SP).
 
 When $a$ is a positive element in a \Cs{} $\cA$ and $\ep >0$, let $(a-\ep)_+$ denote the positive part of the self-adjoint element $a - \ep \, 1_\cA$. 

\begin{proposition} \label{prop:C*-irr} A unital inclusion $\cB \subseteq \cA$ of \Cs s is $C^*$-irreducible if and only if each non-zero positive element $\cA$ is full relatively to $\cB$.

If $\cA$ has property (SP), then it suffices to verify that each non-zero projection in $\cA$ is full relatively to $\cB$.
\end{proposition}

\begin{proof} Take an intermediate \Cs{} $\cB \subseteq \cD \subseteq \cA$. 
If each non-zero positive element $\cA$ is full relatively to $\cB$, then each non-zero positive element of $\cD$ is full in $\cD$, whence $\cD$ is simple. Suppose, conversely, that $\cB \subseteq \cA$ is a $C^*$-irreducible inclusion. Let $a \in \cA$ be a non-zero positive element. Then $a$ is full in the (necessarily simple) intermediate \Cs{} $C^*(\cB,a)$, so $a$ is full relatively to $\cB$ by Proposition~\ref{prop:rel-full}.

Suppose finally that $\cA$ has property (SP). Let $a$ be a non-zero positive element in $\cA$ and choose $0 < \ep < \|a\|$. Then $(a-\ep)_+$ is non-zero and positive, so the hereditary sub-\Cs{} $\overline{(a-\ep)_+\cA (a-\ep)_+}$ contains a non-zero projection $p$, which by assumption is full relatively to $\cB$. As $a \ge \ep p$, it follows that $a$ is full relatively to $\cB$.
\end{proof}

\begin{remark} \label{rem:irr}
If $\cB \subseteq \cA$ is a $C^*$-irreducible inclusion of \Cs s, then 
$\cB' \cap \cA = \C$, i.e., the inclusion is irreducible. Indeed, if $\cB' \cap \cA \ne \C$ and if $a$ is a positive non-zero and non-invertible element in $\cB' \cap \cA$, then $a$ is not full in $C^*(\cB,a)$, cf.\ Proposition~\ref{prop:rel-full} and Lemma~\ref{lm:rel-full}, and $C^*(\cB,a)$ is therefore is a non-simple intermediate \Cs. (More generally,  $C^*(\cB,a)$ is non-simple whenever $a \in \cB' \cap \cA$ is non-zero and non-invertible.) 

In particular, if $\mathcal{E}$ and $\cB$ are unital simple \Cs s with $\cB \ne \C$, then $\mathcal{E} \otimes \C \subseteq \mathcal{E} \otimes \cB$ is not irreducible and hence not $C^*$-irreducible. Also, there are no non-trivial $C^*$-irrecucible inclusions of finite dimensional \Cs s.

The converse does not hold, even under the additional (necessary) assumption that $\cA$ and $\cB$ are simple, see Example~\ref{ex:irr-nirr}. For another example,  any irreducible inclusion $\mathcal{N} \subseteq \mathcal{M}$ of II$_1$ factors with infinite Jones index fails to be $C^*$-irreducible, cf.\ Theorem~\ref{thm:vNirr} below.
\end{remark}

\begin{remark} Neither irreducibility nor $C^*$-irreducibility are ``transitive'' in the sense that  if $\mathcal{C} \subseteq \cB \subseteq \cA$ are unital inclusions of \Cs s such that $\mathcal{C} \subseteq \cB$ and $\cB \subseteq \cA$ are irreducible, respectively, $C^*$-irreducible, then one cannot conclude that $\mathcal{C} \subseteq \cA$ has the same property.

Consider, for an example, an irreducible inclusion $\mathcal{N} \subseteq \mathcal{M}$ of II$_1$ factors with finite Jones index, and let $\mathcal{M}_1 = (\mathcal{M} \cup \{e_N\})''$ be the standard construction of Jones. Then $\mathcal{M} \subseteq \mathcal{M}_1$ is irreducible and $[\mathcal{M}_1:\mathcal{M}] = [\mathcal{M} : \mathcal{N}]<\infty$. Hence $\mathcal{N} \subseteq \mathcal{M}$  and $\mathcal{M} \subseteq \mathcal{M}_1$  are both $C^*$-irreducible by Theorem~\ref{thm:vNirr}, but $\mathcal{N}' \cap \mathcal{M}_1 \ne \C 1$ (since $e_N \in \mathcal{N}' \cap \mathcal{M}_1$), so $\mathcal{N} \subseteq \mathcal{M}_1$ is not irreducible, and hence not $C^*$-irreducible.

Here is another example: Take a unital simple \Cs{} $\cB$ and an outer action $\alpha$ of a cyclic group $\Z_d$ on $\cB$. Let $\widehat{\alpha}$ be the dual action of $\widehat{\Z/d}$ on $\cB \rtimes_\alpha \Z_d$. Then $\cB \subseteq \cB \rtimes_\alpha \Z_d$ and $\cB \rtimes_\alpha \Z_d \subseteq (\cB \rtimes_\alpha \Z_d) \rtimes_{\widehat{\alpha}} \widehat{\Z_d}$ are $C^*$-irreducible by Theorem~\ref{thm:crossed-product}. However, by Takai duality the inclusion $\cB \subseteq (\cB \rtimes_\alpha \Z_d) \rtimes_{\widehat{\alpha}} \widehat{\Z_d}$ is conjugate to the inclusion $\cB \otimes 1_d \subseteq \cB \otimes M_d$, which is not $C^*$-irreducible, cf.\ Remark~\ref{rem:irr} above. 
\end{remark}

\noindent The proof of the second part of the lemma below is essentially identical with the proof of \cite[Proposition 2.2]{Ror:UHFII}. 

\begin{lemma} \label{lm:C*-irr} Let $\cB \subseteq \cA$ an inclusion of unital \Cs s. Suppose that $\cB$ is simple, and let $a \in \cA$ be positive. Then $a$ is full in $\cB$ if there exist $x \in \cB$ and a positive non-zero element $b \in \cB$ such that either $b \le x^*ax$ or $\|b-x^*ax\| < \|b\|$.
\end{lemma}

\begin{proof} Note first that $b$ is full in $\cB$ by simplicity of $\cB$, so if $b \le x^*ax$, then $x^*ax$, and hence $a$, are full relatively to $\cB$. 

Suppose that $\|b-x^*ax\| < \|b\|$.  Set $\delta= \|b-x^*ax\|$, and choose $\delta < \ep < \|b\|$. Choose a continuous function $\varphi \colon \R^+ \to \R^+$ which is zero on $[0,\delta]$ and $1$ on $[\ep,\infty)$. Then
$$(b-\ep)_+ = \varphi(b) (b-\ep)_+ \varphi(b) \le \varphi(b)(b - \delta \cdot 1_\cA)\varphi(b) \le \varphi(b)x^*ax\varphi(b).$$
The latter claim now follows from the former.
\end{proof}

\noindent We conclude this section by reviewing some related properties of unital inclusions of \Cs s:

\begin{definition} \label{def:RDP} A unital inclusion $\cB \subseteq \cA$ of \Cs s is said to have the \emph{relative Dixmier property} if $C_\cB(a) \cap \C \! \cdot \! 1_\cB \ne \emptyset$, for all $a \in \cA$, where $C_\cB(a)$ denotes the closure of the convex hull of $\{u^*au : u \in \cU(\cB)\}$.
\end{definition}

\noindent 
It was shown in \cite{Haagerup-Zsido} that a  unital simple \Cs{} $\cA$ has the Dixmier property (i.e., $C_\cA(a) \cap \C \! \cdot \! 1_\cA \ne \emptyset$, for all $a \in \cA$), if and only if $\cA$ has at most one tracial state. 

Clearly, if  $\cB \subseteq \cA$ has the relative Dixmier property, then both $\cA$ and $\cB$ have the Dixmier property. Moreover,  if $\cA$ has a tracial state $\tau$, then this trace must be unique, its restriction to $\cB$ is a unique trace on $\cB$, and  $C_\cB(a) \cap \C \! \cdot \! 1_\cB = \{\tau(a) \! \cdot \! 1_\cB\}$, for all $a \in \cA$.

Popa proved in \cite[Theorem 2.1]{Popa:JFA2000} that $\cB \subseteq \cA$ has the relative Dixmier property if (i) $\cB$ has the Diximier property, (ii) the inclusion has finite index with respect to some conditional expectation $E \colon \cA \to \cB$ (i.e., there exists $\lambda >0$ such that $E(a) \ge \lambda a$, for all $a \in \cA^+$), and (iii) $\pi_\varphi(\cB)' \cap \pi_\varphi(\cA)'' = \C$, for some state $\varphi$ on $\cA$ (and where $\pi_\varphi$ is the associated GNS representation). Condition (i) is also necessary, but condition (ii) is not (see \cite[Corollary 4.1]{Popa:JFA2000} and also Theorem~\ref{thm:crossed-product}). If $\cB \subseteq \cA$ has the relative Dixmier property, then $\cB' \cap \cA = \C$, which is weaker than condition (iii).

Observe that $\cB \subseteq \cA$ has the relative Dixmier property if and only if $C_\cB(a) \cap \cB \ne \emptyset$, for all $a \in \cA$, and $\cB$ has the Dixmier property. Neither of these two properties hold in general for $C^*$-irreducible inclusions $\cB \subseteq \cA$. Indeed, one can construct $C^*$-irreducible inclusions $\cB \subseteq \cA$ where $\cB$ has unique trace and $\cA$ has more than one trace (use, e.g., Theorem~\ref{thm:AK} or Theorem~\ref{thm:crossed-product}), and for such inclusions we must have $C_\cB(a) \cap \cB \ne \emptyset$, for some $a \in \cA$; and one can use Theorem~\ref{thm:crossed-product} to construct the inclusion such that  $\cB$ has more than one trace, in which case $\cB$ does not have the Dixmier property.
It follows easily from Lemma~\ref{lm:rel-full} and Proposition~\ref{prop:C*-irr} that $\cB \subseteq \cA$ is $C^*$-irreducible if and only if $C_\cB(a)$ contains an invertible element (in $\cA$) for each non-zero positive element $a \in \cA$. This proves the following:

\begin{proposition} \label{prop:RDP-a}
A unital inclusion $\cB \subseteq \cA$ of \Cs s is $C^*$-irreducible if it
has the relative Dixmier property and $\cA$ has a faithful tracial state.
\end{proposition}

\begin{definition} \label{def:PP} 
A unital inclusion $\cB \subseteq \cA$ of \Cs s with a conditional expectation $E \colon \cA \to \cB$ is said to have the \emph{pinching property} if for each non-zero positive element $a \in \cA$ and each $\ep >0$ there exists a contraction $h \in \cB$ such that $\|h^*ah - h^*E(a)h\| \le \ep$ and $\|h^*E(a)h\| \ge \|E(a)\| -\ep$.
\end{definition}

\begin{definition} \label{def:REP}
A unital inclusion $\cB \subseteq \cA$ of \Cs s is said to have the \emph{relative excision property} with respect to a state $\psi$ on $\cA$ if there is a net $\{h_\alpha\}$ of positive elements in $\cB$ satisfying  $\|h_\alpha\|=1$ and $\lim_\alpha \|h_\alpha^{1/2} a h_\alpha^{1/2} - \psi(a) h_\alpha\| = 0$, for all $a \in \cA$.
\end{definition}

\noindent Kwasniewski and Meyer consider in  \cite{Kwas-Meyer:aperiodic}  a related version of the pinching property  leading to their notion of \emph{aperiodic inclusions}, meaning an inclusion $\cB \subseteq \cA$ of \Cs s for which for each $a \in \cA$, each non-zero hereditary sub-\Cs{} $H$ of $\cA$, and each $\ep >0$ there exists 
$x \in H^+$ with $\|x\|=1$ and $b \in \cB$ such that $\|xax-b\|<\ep$. They show, under some additional assumption on the inclusion, that aperiodicity of $\cB \subseteq \cA$ implies that $\cB$ ``detects ideals'' of all intermediate \Cs s of the inclusions, which in particular implies $C^*$-irreducibility when $\cB$ is simple, see \cite[Theorems 7.2 and 7.3]{Kwas-Meyer:aperiodic}. 

Definition~\ref{def:REP} is a relative version of the usual excision property for a state $\psi$ on a \Cs{} $\cA$, in which the net $\{h_\alpha\}$ resides in $\cA$.  The excision property is known to hold for all pure states, and more generally for all states in the weak$^*$ closure of the pure states, and hence for all states if the \Cs{} is antiliminal, see \cite{AkeAndPed:excising}.

The relative excision property behaves well with respect to tensor products, see Proposition~\ref{prop:tensorIII}, and holds in some naturally occuring cases, see, e.g., Lemma~\ref{lm:crossed-product}. It would be interesting to understand which states satisfy the relative excision property with respect to a given inclusion $\cB \subseteq \cA$.

\begin{proposition} \label{prop:M}
Let $\cB \subseteq \cA$ be a unital inclusion of \Cs s which either has the pinching property with  respect to some faithful conditional expectation $E \colon \cA \to \cB$, or has the relative excision property with respect to a faithful state $\psi$ on $\cA$. If, in addition, $\cB$ is simple, then $\cB \subseteq \cA$ is $C^*$-irreducible.
\end{proposition}

\begin{proof} Assume first $\cB \subseteq \cA$ has the pinching property with respect to a faithful conditional expectation $E \colon \cA \to \cB$. Let $a \in \cA$ be non-zero and positive. Then we can find a contraction $h \in \cB$ such that $\|h^*ah - h^*E(a)h\| < \frac12 \|E(a)\| \le \|h^*E(a)h\|$. Hence $a$ is full relatively to $\cB$, by Lemma~\ref{lm:C*-irr}.

Assuming instead that $\cB \subseteq \cA$ has the relative excision property, for $a \in \cA$ non-zero and positive, take $h \in \cB^+$ with $\|h\|=1$ such that $\|h^{1/2} a h^{1/2} - \psi(a) h\| <  \|\psi(a)h\|$. Then, again by Lemma~\ref{lm:C*-irr}, we conclude that $a$ is full relatively to $\cB$.
\end{proof}

\noindent Izumi proved in \cite[Theorem 3.3]{Izumi:Crelle2002} that if $\cB \subseteq \cA$ is a unital inclusion of \Cs s with a conditional expectation $E \colon \cA \to \cB$ of finite index, and if $\cA$, respectively, $\cB$  is simple, then $\cB$, respectively, $\cA$, is a finite direct sum of simple \Cs s. In particular, if the inclusion $\cB \subseteq \cA$ is irreducible, then $\cA$ is simple if and only if $\cB$ is simple. This proves the following:

\begin{corollary}[Izumi] A unital  inclusion $\cB \subseteq \cA$  of \Cs s of finite index with respect to some conditional expectation $E \colon \cA \to \cB$ is $C^*$-irreducible if and only if it is irreducible.
\end{corollary}

\section{Irreducible inclusions of von Neumann algebras} \label{sec:vNalg}

\noindent An inclusion $\cN \subseteq \cM$ of von Neumann factors is \emph{irreducible} if $\cN' \cap \cM = \C$. Since $\cP' \cap \cP \subseteq \cN' \cap \cM = \C$ for each intermediate von Neumann algebra $\cN \subseteq \cP \subseteq \cM$, each such von Neumann algebra $\cP$ is a factor. This analogy with $C^*$-irreducible inclusions goes further as shown in the remark below and in Theorem~\ref{thm:vNirr}.

\begin{remark}
An inclusion $\cN \subseteq \cM$ of von Neumann factors is irreducible if and only if $\bigvee_{u \in \cU(\cN)}  u^*pu = 1$ for each non-zero projection $p \in \cM$. Indeed, $\bigvee_{u \in \cU(\cN)}  u^*pu$ is easily seen to belong to $\cN' \cap \cM$; and if $p \in \cN' \cap \cM$, then $\bigvee_{u \in \cU(\cN)}  u^*pu = p$.

By Proposition~\ref{prop:C*-irr} and Lemma~\ref{lm:rel-full}, $\cN \subseteq \cM$ is $C^*$-irreducible if and only if for each non-zero projection $p \in \cM$ there exist finitely many unitaries $u_1, \dots, u_n \in \cN$ such that $\sum_{j=1}^n u_j^*pu_j \ge 1$. Now,  $\sum_{j=1}^n u_j^*pu_j \ge 1$ implies $\bigvee_{j=1}^n u_j^*pu_j = 1$, while $\bigvee_{j=1}^n u_j^*pu_j = 1$  does not even imply that $\sum_{j=1}^n u_j^*pu_j$ is invertible.  Nonetheless, the following question may still have a positive answer.
\end{remark}

\begin{question} Is an inclusion $\cN \subseteq \cM$ of von Neumann factors $C^*$-irreducible if and only if for each non-zero projection $p \in \cM$ there exist finitely many unitaries $u_1, \dots, u_n \in \cN$ such that $\bigvee_{j=1}^n u_j^*pu_j = 1$?
\end{question}

\begin{definition} A state $\phi$ on a von Neumann algebra $\cM$ is \emph{singular} if for each non-zero projection $p \in \cM$ there exists a non-zero projection $q \le p$ in $\cM$ such that $\phi(q)=0$.
\end{definition}

\noindent
A maximality argument shows that if $\phi$ is a singular state on $\cM$, then, for each projection $p \in \cM$, there is a family $(p_i)_{i \in I}$ of pairwise orthogonal non-zero projections in $\cM$ satisfying $p = \sum_{i \in I} p_i$ and $\phi(p_i)=0$, for all $i \in I$. Hence the restriction of $\phi$ to each corner $p\cM p$ is non-normal.

The theorem below is essentially a restatement of the main result from Popa's paper \cite{Popa:ENS1999}. 
I thank Adrian Ioana for pointing this out to me and also for explaining how to obtain  (i) $\Rightarrow$ (iii)  via results of Pop, \cite{FPop:PAMS1998}, and Popa, \cite{Popa:ENS1999}.

\begin{theorem}[Popa] \label{thm:vNirr}
The following conditions are equivalent for any inclusion $\cN \subseteq \cM$ of  II$_1$-factors with separable predual.
\begin{enumerate}
\item $\cN \subseteq \cM$ is $C^*$-irreducible,
\item $\cN \subseteq \cM$ has the relative Dixmier property,
\item $\cN \subseteq \cM$ is irreducible with finite Jones index.
\end{enumerate}
\end{theorem}

\begin{proof} The equivalence of (ii) and (iii) is \cite[Theorem 2.1]{Popa:ENS1999} by Popa. 

 (i) $\Rightarrow$ (iii).
 Suppose that $\cN \subseteq \cM$ is $C^*$-irreducible. Then $\cN' \cap \cM = \C$, as observed in Remark~\ref{rem:irr}. We must show that $[\cM:\cN] < \infty$.
 
 Let $\tau$ denote the tracial state on $\cM$ and on $\cN$. Following the proof of Pop, \cite[Proposition 3.3]{FPop:PAMS1998}, we first show that there is no singular state $\phi$ on $\cM$ which extends the trace on $\cN$. Indeed, as shown in \cite[Proposition 3.3]{FPop:PAMS1998}, if such a state $\phi$ exists, then there is a singular positive $\cN$-bimodular map $E \colon \cM \to L^2(\cN,\tau)$ extending the identity map on $\cN$.  Being singular implies that there exists a non-zero projection $e \in \cM$ with $E(e)=0$.  Since $\cN \subseteq \cM$ is $C^*$-irreducible we can find unitaries $u_1, \dots, u_n$ in $\cN$ such that $\sum_{j=1}^n u_j^*eu_j \ge 1_\cM$. This leads to the contradiction: 
$$1_\cM = E(1_\cM) \le E(\sum_{j=1}^n u_j^*eu_j) = \sum_{j=1}^n u_j^*E(e)u_j = 0.$$
On the other hand, Popa shows in \cite{Popa:ENS1999} (p.~763) that if $[\cM:\cN]=\infty$, then there exist a singular state $\phi$ on $\cM$ that extends the trace on $\cN$, so we conclude that  $[\cM:\cN] <\infty$.

We include for completeness a brief sketch of Popa's argument from \cite{Popa:ENS1999}. It was shown in \cite{PimPop:1986} that $$[\cM:\cN] ^{-1} = \sup \{\lambda \ge 0: E_\cN(x) \ge \lambda x, \; \forall x \in \cM^+\},$$ where $E_\cN$ is the canonical trace preserving conditional expectation onto $\cN$. Assuming, to reach a contradiction, that $[\cM:\cN]=\infty$, Popa constructs, for each $n \ge 1$, a positive element $a_n \in \cM$ such that $E_\cN(a_n) = 1_\cN$ and $\tau(s(a_n)) \le 2^{-n}$, where $s(a_n)$ denotes the support projection of $a_n$. Define, for each $n \ge 1$, a (normal) state $\psi_n$ on $\cM$ by $\psi_n(x) = \tau(xa_n)$, $x \in \cM$. Then $\psi_n(y) = \tau(ya_n) = \tau(yE_\cN(a_n)) = \tau(y)$, for all $y \in \cN$.

Let $\psi$ be a weak$^*$ accumulation point of the sequence $\{\psi_n\}$. Then $\psi(y) = \tau(y)$, for all $y \in \cN$. Set $p_m = \bigvee_{n=m}^\infty s(a_n)$. Then $\{p_m\}$ is a decreasing sequence of projections in $\cM$ with limit $\bigwedge_{n=1}^\infty p_n = 0$, because $$\tau\big(\bigwedge_{n=1}^\infty p_n\big) = \inf_n \tau(p_n) \le \inf_n \sum_{j=n}^\infty \tau(s(a_j)) \le \inf_n 2^{-n+1} = 0.$$ 
On the other hand, $\psi_n(s(a_n)) = \tau(s(a_n)a_n) = \tau(a_n) = 1$, for all $n \ge 1$, so $\psi_n(p_m) = 1$, for all $n \ge m$, which implies that $\psi(p_m)=1$, for all $m \ge 1$. Thus $1-p_n \to 1_\cM$ while $\psi(1-p_n)=0$, which shows that $\psi$ is singular.

(ii) $\Rightarrow$ (i) holds by Proposition~\ref{prop:RDP-a}. 
\end{proof}

\noindent The condition of finite index appearing in Theorem~\ref{thm:vNirr} above does not carry over to general unital inclusions of simple \Cs s $\cB \subseteq \cA$ as noted in \cite[Corollary 4.5]{Popa:ENS1999}, see also 
Example~\ref{ex:FI}, Theorem~\ref{thm:crossed-product} and Proposition~\ref{prop:tensorIII}.

\begin{remark} It follows from Popa's theorem (the equivalence of (ii) and (iii) in the theorem above) that if $\cN_i \subseteq \cM_i$, $i=1,2$, are inclusions of II$_1$-factors with the relative Dixmier property, then $\cN_1 \overline{\otimes} \cN_2 \subseteq \cM_1 \overline{\otimes} \cM_2$ also has the relative Dixmier property; and further that this fails for infinite tensor products.
\end{remark}

\begin{question} Is each irreducible inclusion $\cN \subseteq \cM$ of type III factors $C^*$-irreducible? 
\end{question}

\section{Inclusions arising from groups and dynamical systems} \label{sec:dynamics}

\noindent Inclusions of both von Neumann algebras and \Cs s arise from groups and dynamical systems in several interesting ways. In Section~\ref{sec:background} we already mentioned results of Amrutam--Kalantar and Amrutam--Ursu stating when inclusions of the form $C^*_\lambda(\Gamma) \subseteq \cA \rtimes_{\mathrm{red}} \Gamma$ and $C(Y) \rtimes_\mathrm{red} \Gamma \subseteq C(X) \rtimes_\mathrm{red} \Gamma$ are $C^*$-irreducible. We shall here add further related examples to the list. At the end of the section we revisit the Amrutam--Kalantar theorem.

\subsection*{Inclusions arising from subgroups}

\noindent We consider in this subsection the case of  an inclusion $\Lambda \subseteq \Gamma$  of discrete groups, which gives rise to inclusions $\mathcal{L}(\Lambda) \subseteq \mathcal{L}(\Gamma)$ and $C^*_\lambda(\Lambda) \subseteq C^*_\lambda(\Gamma)$ 
of the associated finite von Neumann algebras, respectively, their reduced group \Cs s.

We say that \emph{$\Gamma$ is an icc group relatively to $\Lambda$} if $\{tst^{-1} : t \in \Lambda\}$ is infinite for all $s \in \Gamma \setminus \{e\}$. This condition implies that both $\Lambda$ and $\Gamma$ are icc, and hence that both $\mathcal{L}(\Lambda)$ and  $\mathcal{L}(\Gamma)$ are II$_1$-factors. We first note the following straightforward result:

\begin{proposition} \label{prop:rel-icc}
The following conditions are equivalent for any inclusion $\Lambda \subseteq \Gamma$ of discrete groups:
\begin{enumerate}
\item $\Gamma$ is an icc group relatively to $\Lambda$,
\item $\mathcal{L}(\Lambda)' \cap \mathcal{L}(\Gamma) = \C$,
\item $C^*_\lambda(\Lambda)' \cap C^*_\lambda(\Gamma) = \C$.
\end{enumerate}
Moreover, $\mathcal{L}(\Lambda) \subseteq \mathcal{L}(\Gamma)$ is $C^*$-irreducible if and only if $\Gamma$ is an icc group relatively to $\Lambda$ and $[\Gamma:\Lambda] < \infty$.
\end{proposition}

\begin{proof} (i) $\Rightarrow$ (ii). If $T \in \mathcal{L}(\Lambda)' \cap \mathcal{L}(\Gamma)$, then the function $s \mapsto (T\delta_e)(s)$, $s \in \Gamma$, is constant on each $\Lambda$-conjugacy class (where $\{\delta_t\}_{t \in \Gamma}$ is the standard orthonormal basis for $\ell^2(\Gamma)$). Hence $(T\delta_e)(s) = 0$, for $s \ne e$. This implies that $T = cI$, where $c = (T\delta_e)(e)$. 

(ii) $\Rightarrow$ (iii) follows from the fact that $\mathcal{L}(\Lambda)' = C^*_\lambda(\Lambda)'$ and $C^*_\lambda(\Gamma) \subseteq \mathcal{L}(\Gamma)$.  (iii) $\Rightarrow$ (i). If $S:= \{t^{-1}st : t \in \Lambda\}$ is finite, for some $s \ne e$, then $a = \sum_{t \in S} u_t$ belongs to  $C^*_\lambda(\Lambda)' \cap C^*_\lambda(\Gamma)$ and $a \notin \C$, where $\lambda \colon t \mapsto u_t$ is the left-regular unitary representation of $\Gamma$ on $\ell^2(\Gamma)$.

The last claim follows from Theorem~\ref{thm:vNirr} combined with the equivalence of (i) and (ii) and the fact that $[\mathcal{L}(\Gamma):\mathcal{L}(\Lambda)] = [\Gamma:\Lambda]$.
\end{proof}

\noindent We proceed to consider  when an inclusion of discrete groups gives rise to a $C^*$-irreducible inclusion of their reduced group \Cs s. This is much more subtle than the corresponding question for von Neumann algebras covered in the proposition above. First, the reduced group \Cs{} of an icc group need not be simple, and second, even when $\Lambda$ and $\Gamma$ both are $C^*$-simple and $\Gamma$ is icc relatively to $\Lambda$, it is still not the case that $C^*_\lambda(\Lambda) \subseteq C^*_\lambda(\Gamma)$ is $C^*$-irreducible, although this is true when $\Lambda$ is normal in $\Gamma$.

We remind the reader of the recently developed deep characterization of $C^*$-simple groups. The breakthrough came with the paper by Breuillard--Kalantar--Kennedy--Ozawa, \cite{BKKO},  where the equivalence of (i) and (ii) below was established (among many other results). It was followed up by  Haagerup, \cite{Haagerup:new-look}, who proved  the equivalence of (ii), (iii) and (iv), which was independently discovered by Kennedy, \cite{Kennedy:C*-simple}, who moreover added conditions (v) and (vi) to the list.


\begin{theorem}[Breuillard--Kalantar--Kennedy--Ozawa, Haagerup, Kennedy] \label{thm:C*-simple}
The following conditions are equivalent for a discrete group $\Gamma$:
\begin{enumerate}
\item $\Gamma$ is $C^*$-simple,
\item $\Gamma$ acts freely on its Furstenberg boundary $\partial_F \Gamma$,
\item $\tau_0$ belongs to the weak$^*$ closure of $\{s.\varphi : s \in \Gamma\}$, for each state $\varphi$ on $C^*_\lambda(\Gamma)$,
\item $\Gamma$ satisfies the Powers' averaging property: for all $\ep >0$ and for all $s_1,\dots, s_n \in \Gamma \setminus \{e\}$ there exist $t_1, \dots, t_m \in \Gamma$ such that $\|\frac{1}{m} \sum_{k=1}^m u_{t_k^{-1}s_jt_k}\| < \ep$, for $j=1, \dots, n$.
\item $\Gamma$ has no non-trivial amenable residually normal  subgroups,
\item $\Gamma$ has no non-trivial amenable uniformly recurrent subgroups.
\end{enumerate}
\end{theorem}

\noindent
Condition (iv) is equivalent to the following more standard formulation of the Powers' averaging property: for all $\ep >0$ and for all $x \in C^*_\lambda(\Gamma)$ there exist $t_1, \dots, t_m \in \Gamma$ such that $\|\frac{1}{m} \sum_{k=1}^m u_{t_k}^* x u_{t_k} - \tau_0(x)\cdot 1\| < \ep$.

A few words about the terminology of the theorem. As before, $\tau_0$ denotes the canonical trace on $C^*_\lambda(\Gamma)$.  If $\varphi$ is a state on $C^*_\lambda(\Gamma)$ and $s \in \Gamma$, then $s.\varphi$ denotes the state $(s.\varphi)(x) = \varphi(u_s^*xu_s)$, $x \in C^*_\lambda(\Gamma)$. 

A subgroup $\Lambda$ of a group $\Gamma$ is \emph{residually normal}\footnote{Group theorists say that $\Lambda$ is a \emph{confined subgroup} in this case.} if there is a finite subset $F$ of $\Gamma \setminus \{e\}$ such that $F \cap t^{-1} \Lambda t \ne \emptyset$, for all $t \in \Gamma$.  Let $\mathrm{Sub}(\Gamma)$ denote the compact Hausdorff space of all subgroups of $\Gamma$  (viewed as a closed subset of the compact Cantor space $\mathcal{P}(\Gamma) = \{0,1\}^\Gamma$ of all subsets of $\Gamma$). The group $\Gamma$ acts on $\mathrm{Sub}(\Gamma)$ by conjugation. A \emph{uniformly recurrent subgroup} is a minimal closed $\Gamma$-invariant subspace $X$ of $\mathrm{Sub}(\Gamma)$. It is amenable if all subgroups in $X$ are amenable, and it is trivial if $X$ is the singleton consiting of the trivial group $\{e\}$. 

Analogous to the situation in Proposition~\ref{prop:rel-icc}, natural relative versions of the conditions in Theorem~\ref{thm:C*-simple} suggest themselves as candidates for ensuring $C^*$-irreducibility of the reduced \Cs s of the inclusion of groups. I thank Mehrdad Kalantar for suggesting the present formulation of (ii), which is an adjustment of our first version of this condition.

\begin{theorem} \label{thm:C*-group} Consider the following conditions for an inclusion $\Lambda \subseteq \Gamma$ of $C^*$-simple groups.
\begin{enumerate}
\item The inclusion $C^*_\lambda(\Lambda) \subseteq C^*_\lambda(\Gamma)$ is $C^*$-irreducible,
\item there is a topologically free boundary action $\Gamma \curvearrowright X$ such that, for each probability measure $\mu$ on $X$, the weak$^*$ closure of the orbit $\Lambda.\mu$ contains a point mass $\delta_x$, for some $x \in X$ on which $\Gamma$ acts freely,
\item $\tau_0$ belongs to the weak$^*$ closure of $\{s.\varphi : s \in \Lambda\}$, for each state $\varphi$ on $C^*_\lambda(\Gamma)$,
\item $\tau_0$ belongs to the weak$^*$ closure of $\mathrm{conv}\{s.\varphi : s \in \Lambda\}$, for each state $\varphi$ on $C^*_\lambda(\Gamma)$,
\item $\Gamma$ has the Powers' averaging property relatively to $\Lambda$:
 for all $\ep >0$ and for all $s_1,\dots, s_n \in \Gamma \setminus \{e\}$ there exist $t_1, \dots, t_m \in \Lambda$ such that $\|\frac{1}{m} \sum_{k=1}^m u_{t_k^{-1}s_jt_k}\| < \ep$, for $j=1, \dots, n$,
\item the inclusion $C^*_\lambda(\Lambda) \subseteq C^*_\lambda(\Gamma)$ has the relative Dixmier property.
\end{enumerate}
Then {\emph{(ii)}} $\Rightarrow$ \emph{(iii)} $\Rightarrow$ \emph{(iv)} $\Leftrightarrow$ \emph{(v)} $\Rightarrow$ \emph{(vi)} $\Rightarrow$ \emph{(i)}, and \emph{(i)} $\Rightarrow$ \emph{(vi)} if $[\Gamma: \Lambda]<\infty$. 
\end{theorem}

\noindent Condition (v) is in the papers \cite{Amrutam:2019} and \cite{Ursu:2020} referred to as $\Lambda$ being a \emph{plump} subgroup of $\Gamma$, and it can equivalently be expressed as 
$$\tau_0(x) \! \cdot \! 1 \in \overline{\mathrm{conv}\{u_t^* x u_t : t \in \Lambda\}},$$
for all $x \in  C^*_\lambda(\Gamma)$.  

In \cite{Ursu:2020}, Ursu proved a number of reformulations of (v)  in the case where $\Lambda$ is \emph{normal} in $\Gamma$, including  that the action $\Gamma \curvearrowright \partial_F \Lambda$ is free\footnote{The action $\Lambda \curvearrowright \partial_F \Lambda$ extends uniquely to an action $\Gamma \curvearrowright \partial_F \Lambda$, when $\Lambda$ is normal in $\Gamma$, cf. \cite[Lemma 21]{Ozawa:Furstenberg}.}. This implies that (ii) holds, and hence 
entails that conditions (ii)--(v) above all are equivalent when $\Lambda$ is normal in $\Gamma$. 

In the work, \cite{Omland}, still under preparation, Tron Omland further proved that, in fact, all conditions (i)--(vi) above are equivalent when $\Lambda$ is normal in $\Gamma$ and, moreover,  that (i) holds if and only if $\Gamma$ and $\Lambda$ are $C^*$-simple and $\Gamma$ is icc relatively to $\Lambda$, i.e., that $C^*_\lambda(\Lambda)' \cap C^*_\lambda(\Gamma) = \C$.  Moreover, if $\Lambda$ is $C^*$-simple and $\Gamma$ is icc relatively to $\Lambda$, then $\Gamma$ is automatically $C^*$-simple. We thus have a complete understanding of when normal inclusions of $C^*$-simple groups give rise to $C^*$-irreducible inclusions of their reduced \Cs s.

Omland proved in the same paper that one can find a (necessarily non-normal and infinite index) inclusion $\Lambda \subseteq \Gamma$ of $C^*$-simple groups such that the inclusion $C^*_\lambda(\Lambda) \subseteq C^*_\lambda(\Gamma)$ is irreducible not $C^*$-irreducible. 

\begin{proof} 
(ii) $\Rightarrow$ (iii). The proof is almost identical to the proof of ``(i) $\Rightarrow$ (ii)'' of \cite[Theorem 4.5]{Haagerup:new-look}. Take a state $\varphi$ on $C^*_\lambda(\Gamma)$, extend it to a state $\psi$ on $C(X) \rtimes_{\mathrm{red}} \Gamma$, and let $\rho$ be the restriction of $\psi$ to $C(X)$. 
By assumption there is a net $\{s_i\}$ in $\Lambda$ such that $s_i.\rho$ converges in the weak$^*$ topology to a point-evaluation $\delta_x$, for some $x \in X$ on which $\Gamma$ acts freely. Upon passing to a subnet we may assume that $\{s_i.\psi\}$ converges to some state $\psi'$ on $C(X) \rtimes_{\mathrm{red}} \Gamma$. Let $\varphi'$ be the restriction of $\psi'$ to $C^*_\lambda(\Gamma)$. Since $\Gamma$ acts freely on $x$ it follows from \cite[Lemma 3.1]{Haagerup:new-look} that $\varphi' = \tau_0$. Hence (iii) holds.

(iii) $\Rightarrow$ (iv) is trivial. (iv) $\Rightarrow$ (v) is a standard Hahn-Banach argument, cf.\ the proof of (iii) $\Rightarrow$ (iv) $\Rightarrow$ (v) of \cite[Theorem 4.5]{Haagerup:new-look}.

(v) $\Rightarrow$ (iv). Let $x_1, \dots, x_n \in C^*_\lambda(\Gamma)$ and $\ep > 0$ be given. Repeated application of (v) shows that there exist $m \ge 1$ and $t_1, \dots, t_m \in \Lambda$ such that
$$\big| \tau_0(x_j) \cdot 1 - m^{-1} \, \sum_{k=1}^m u_{t_k}^* x_j u_{t_k} \big| < \ep,$$
for $j=1, \dots, n$. It follows that $|\tau_0(x_j) - m^{-1} \sum_{k=1}^m (t_k.\varphi)(x_j)| < \ep$, for all states $\varphi$ on $C^*_\lambda(\Gamma)$. This proves that (iv) holds.

(v) $\Rightarrow$ (vi) is trivial, and (vi) $\Rightarrow$ (i)  follows from Proposition~\ref{prop:M}, since $C^*_\lambda(\Gamma)$ has a faithful trace. 

For the last claim, suppose that (i) holds and that $[\Gamma : \Lambda]< \infty$. Then the canonical conditional expectation $E \colon C^*_\lambda(\Gamma) \to C^*_\lambda(\Lambda)$ has finite index. Moreover, (i) implies that  $\mathcal{L}(\Lambda)' \cap \mathcal{L}(\Gamma) = \C$, cf.\ Proposition~\ref{prop:rel-icc} and Remark~\ref{rem:irr}. Lastly, $C^*_\lambda(\Lambda)$ has the Dixmier property (being simple with a unique tracial state), so it follows from \cite[Theorem 2.1]{Popa:JFA2000} (cf.\ the comments below Definition~\ref{def:RDP}) that (vi) holds.
\end{proof}

\begin{example} \label{ex:FI}
For each (non-empty) index set $I$, let $\mathbb{F}_I$ denote the free group with generating set $I$. If $I \subset J$ and $|I| \ge 2$, then the inclusion $C^*_\lambda(\mathbb{F}_I) \subseteq C^*_\lambda(\mathbb{F}_J)$ is $C^*$-irreducible. Indeed, Powers proved in his influential paper \cite{Powers:F_2} that $C^*_\mathrm{\lambda}(\mathbb{F}_I)$ is simple whenever $|I| \ge 2$ (he actually only considered the case $|I| = 2$, but the general result clearly follows from his proof). Theorem 1 of \cite{Powers:F_2} precisely states that (iv) of Theorem~\ref{thm:C*-simple} holds for $\Gamma = \mathbb{F}_2$. Inspection of his proof of this theorem (via Lemmas 3--6) shows that condition (v) of Theorem~\ref{thm:C*-group} holds with $\Gamma = \mathbb{F}_J$  and $\Lambda = \mathbb{F}_I$, as long as $|I| \ge 2$.
\end{example}

\noindent Further examples of   inclusions of $C^*$-simple groups $\Lambda \subseteq \Gamma$ for which $\Lambda$ is a plump subgroup of $\Gamma$, and hence $C^*_\lambda(\Lambda) \subseteq C^*_\lambda(\Gamma)$ is $C^*$-irreducible, can be found in \cite[Section 5]{Ursu:2020}. It would be interesting to expand the dictionary of such inclusions of groups even further.

\begin{question} Which of the ``missing'' implications in Theorem~\ref{thm:C*-group} hold? 
\end{question}

\noindent Condition (i) of Theorem~\ref{thm:C*-group} does not imply that $[\Gamma : \Lambda] < \infty$, and we do not know if the implication (i) $\Rightarrow$ (iv) holds in general without assuming finite index.

\begin{question} \label{q:W} 
  For which ($C^*$-irreducible) inclusions $C^*_\lambda(\Lambda) \subseteq C^*_\lambda(\Gamma)$ is it the case that all intermediate \Cs s are of the form $C^*_\lambda(\Pi)$ for some intermediate group $\Lambda \subseteq \Pi \subseteq \Gamma$? 
\end{question}

\subsection*{Inclusions arising from crossed products}

\noindent An action $\Gamma \curvearrowright \cB$ of a group $\Gamma$ on a simple \Cs{} $\cB$ gives in a canonical way rise to two  inclusions of \Cs s, namely $\cB \subseteq  \cB \rtimes_{\mathrm{red}} \Gamma$ and $C^*_\lambda(\Gamma) \subseteq \cB \rtimes_{\mathrm{red}} \Gamma$. The first class of inclusions is very well understood, but for  completeness of the exposition we review when such inclusions are $C^*$-irreducible and what we know about their intermediate \Cs s.  Part (i) of the following lemma is contained in \cite[Lemma 7.1]{OlePed:C*-dynamicIII}, see also \cite{Kis:Auto}. Recall that an automorphism on a unital \emph{simple} \Cs{} is properly outer if it is not inner.

\begin{lemma} \label{lm:crossed-product}
Let $\cB$ be a unital simple infinite dimensional \Cs, and let $\Gamma$ be a discrete group with an outer action on $\cB$. Let $E \colon   \cB \rtimes_{\mathrm{red}} \Gamma \to \cB$ be the canonical conditional expectation.
\begin{enumerate}
\item For each non-zero positive element $a \in \cB \rtimes_{\mathrm{red}} \Gamma$ and each $\ep >0$ there exists a positive contraction $h \in \cB$ such that $\|hah - hE(a)h\| \le \ep$ and $\|hE(a)h\| \ge \|E(a)\| - \ep$. In particular, $\cB \subseteq \cB \rtimes_{\mathrm{red}} \Gamma$ has the pinching property (Definition~\ref{def:PP}) with respect to  $E$.
\item For each non-zero hereditary sub-\Cs{} $\cB_0 \subseteq \cB$, for each finite subset $\mathcal{F} \subseteq \cB$, and for each $\ep >0$  there exists a positive element $h \in \cB_0$ with $\|h\|=1$ and $\|h(a-E(a))h\| \le \ep$, for all $a \in \mathcal{F}$. 
\item The inclusion $\cB \subseteq \cB \rtimes_{\mathrm{red}} \Gamma$ has the relative excision property (cf.\ Definition~\ref{def:REP}) with respect to each state $\psi$ on $\cB \rtimes_{\mathrm{red}} \Gamma$ that factors through $E$.
\end{enumerate}
\end{lemma}

\begin{proof}
 Let $t \mapsto u_t$, $t \in \Gamma$, denote the unitary representation of $\Gamma$ in the crossed product 
$\cB \rtimes_{\mathrm{red}} \Gamma$. Then $\cA_0 := \mathrm{span} \{\cB u_t : t \in \Gamma\}$ is dense in $\cB \rtimes_{\mathrm{red}} \Gamma$. To see that (i) follows from \cite[Lemma~7.1]{OlePed:C*-dynamicIII}, note that we may assume that $a \in \cA_0$.

(ii). We may assume that $\mathcal{F} \subseteq \cA_0$. Moreover, by replacing each $a \in \mathcal{F}$ with $a-E(a)$, we may further assume that $E(a) = 0$. In the case where $\mathcal{F} = \{a_1\}$ is a singleton, we can use \cite[Lemma 7.1]{OlePed:C*-dynamicIII} (with $a_0$ any non-zero positive element in $\cB_0$) to find positive elements $h_1, h_1' \in \cB_0$ with $\|h_1\|=\|h_1'\|=1$ such that $h_1h_1'=h_1'$ and $\|h_1a_1h_1\| \le\ep$. If $\mathcal{F} = \{a_1,a_2\}$, then take a positive elements $h_2, h_2'$ in $\overline{h_1'\cB h_1'}$ with $\|h_2\|=\|h_2'\|=1$ such that $h_2h_2'=h_2'$ and $\|h_2a_2h_2\| \le \ep$, and note that we still have $\|h_2a_1h_2\| \le \ep$. Continue like this until all elements in $\mathcal{F}$ have been exhausted.

(iii). Let $\ep >0$ and a finite subset $\mathcal{F}$ of $\cA$ be given. Write $\psi = \rho \circ E$ for some state $\rho$ on $\cB$. Since $\cB$ is simple and infinite dimensional, $\rho$ can be excised. It follows that there is a non-zero hereditary sub-\Cs{} $\cB_0 \subseteq \cB$ such that $\|h^{1/2}E(a)h^{1/2} - \rho(E(a))h\| \le \ep/2$ for all positive $h$ in $\cB_0$ with $\|h\|=1$, and for all $a \in \mathcal{F}$.  

By (ii) we can find a positive element $h \in \cB_0$ such that $\|h^{1/2}(a-E(a))h^{1/2}\| \le \ep/2$, for all $a \in \mathcal{F}$, and $\|h\|=1$. It follows that $\|h^{1/2}ah^{1/2} - \psi(a)h\| \le \ep$, for all $a \in \mathcal{F}$. This proves that $\psi$ has the relative excision property as desired.
\end{proof}

\noindent The theorem below is essentially a corollary to \cite[Lemma 7.1]{OlePed:C*-dynamicIII} (= Lemma~\ref{lm:crossed-product} (i) above). The implication (ii) $\Rightarrow$ (i) also follows from Theorem~\ref{thm:PICS} below. It is curious that
(i) and (iii) are equivalent, while this is not the case in the situation of Theorem~\ref{thm:AK} below, cf.\ Example~\ref{ex:irr-nirr}.

\begin{theorem} \label{thm:crossed-product}
Let $\cB$ be a unital simple \Cs, and let $\Gamma$ be a discrete group acting on $\cB$. Then the following conditions are equivalent:
\begin{enumerate}
\item $\cB \subseteq \cB \rtimes_{\mathrm{red}} \Gamma$ is $C^*$-irreducible,
\item the action $\Gamma \curvearrowright \cB$ is outer,
\item $\cB' \cap (\cB \rtimes_{\mathrm{red}} \Gamma)= \C$.
\end{enumerate}
Moreover, if $\cB$ has the Dixmier property (i.e., has at most one tracial state), then each of the equivalent conditions above implies that $\cB \subseteq \cB \rtimes_{\mathrm{red}} \Gamma$ has the relative Dixmier property.
\end{theorem}

\begin{proof} (ii) $\Rightarrow$ (i) follows from Lemma~\ref{lm:crossed-product} and Lemma~\ref{lm:C*-irr}, and (i) $\Rightarrow$ (iii) holds for all inclusions, cf.\ Remark~\ref{rem:irr}. (iii) $\Rightarrow$ (ii). Denote the action $\Gamma \curvearrowright \cB$ by $\alpha$, and suppose that $\alpha_t$ is inner, for some $t \ne e$. Then there is a unitary $u \in \cB$ such that $uxu^* = \alpha_t(x) = u_t x u_t^*$, for $x \in \cB$. It follows that $u^*u_t \in (\cB \rtimes_{\mathrm{red}} \Gamma) \cap \cB'$, and $u^*u_t \notin \C$ by construction of the crossed product \Cs.

The last part of the theorem follows from \cite[Corollary 4.1]{Popa:JFA2000}.
\end{proof}

\noindent As mentioned in the introduction, there is a Galois correspondence between intermediate \Cs s of the inclusions considered in Theorem~\ref{thm:crossed-product} above and subgroups of $\Gamma$, cf.\ the theorem below, which is due to Izumi, \cite{Izumi:Crelle2002} in the case where $\Gamma$ is finite, and to Cameron-Smith, \cite{CamSmith:Galois}, in the general case. 

\begin{theorem}[Izumi--Cameron-Smith]  \label{thm:PICS}
Let $\cB$ be a unital simple \Cs, and let $\Gamma$ be a discrete group acting outerly on $\cB$. Then each intermediate \Cs{} $\cD$ of the inclusion $\cB \subseteq \cB \rtimes_{\mathrm{red}} \Gamma$ is of the form $\cD = \cB \rtimes_{\mathrm{red}} \Lambda$, for some subgroup $\Lambda$ of $\Gamma$.
\end{theorem}

\noindent With his permission we give a simple proof, due to Sorin Popa, of this theorem in the case where $\cB$ has the Dixmier property (i.e., has at most one tracial state). The proof uses the following lemma, of independent interest, that is embedded in the proof of \cite[Corollary 4.1]{Popa:JFA2000}:

\begin{lemma}[Popa] \label{lm:Popa}
Let $\cB$ be a unital simple \Cs{} with the Dixmier property, let $\alpha_1, \dots, \alpha_n$ be outer automorphisms on $\cB$, let $b_1, \dots, b_n \in \cB$, and let $\ep >0$. Then there exist $m \ge 1$ and unitaries $v_1, \dots, v_m \in \cB$ such that
$$\Big\| \frac{1}{m} \sum_{j=1}^m v_j b_i \alpha_i(v_j)^*\Big\|<\ep, \qquad i = 1,2, \dots, n.$$
\end{lemma}

\noindent It seems likely that this lemma holds more generally for all unital simple \Cs s (without the assumption of the Dixmier property), in which case the proof below would hold in this general case as well.

\bigskip
\noindent
{\it{Proof of Theorem~\ref{thm:PICS}}} (in the case where $\cB$ has the Dixmier property).  Denote the action of $\Gamma$ on $\cB$ by $\alpha$, and let $t \mapsto u_t$, $t \in \Gamma$, be the unitary representation of $\Gamma$ in the crossed product. 

For each $x \in  \cB \rtimes_{\mathrm{red}} \Gamma$, written formally as $x = \sum_{t \in \Gamma} a_t u_t$, with $a_t \in \cB$, let $\mathrm{supp}(x)$ be the set of those $t \in \Gamma$ for which $a_t \ne 0$. It suffices to show that $u_s \in C^*(\cB,x)$ whenever $x \in  \cB \rtimes_{\mathrm{red}} \Gamma$ and $s \in \mathrm{supp}(x)$. Indeed, if $\cD$ is an intermediate \Cs, then let $\Lambda$ be the subgroup of $\Gamma$ spanned by $\bigcup_{x \in \cD} \mathrm{supp}(x)$. It then easily follows from the claim above that $\cD = \cB \rtimes_{\mathrm{red}} \Lambda$. 

Take $x \in \cB \rtimes_{\mathrm{red}} \Gamma$, $s \in \mathrm{supp}(x)$, and $\ep >0$. As above, write $x = \sum_{t \in \Gamma} a_t u_t$ with $a_s \ne 0$. By simplicity of $\cB$, which implies that $\cB$ is algebraically simple, there are $n \ge 1$ and elements $c_1, \dots, c_n, d_1, \dots, d_n \in \cB$ such that $1 = \sum_{i=1}^n c_ia_sd_i$. Set $y = \sum_{i=1}^n c_i x \alpha_s^{-1}(d_i)$. Then $y \in C^*(\cB,x)$, and $y = \sum_{t \in \Gamma} b_t u_t$, with $b_t \in \cB$ and $b_s=1$. 

Choose $y_0 \in \cB \rtimes_{\mathrm{red}} \Gamma$ such that 
$\|y-y_0\| \le \ep/3$ and such that $F:=\mathrm{supp}(y_0)$ is finite. Write $y_0 = \sum_{t \in F} b'_tu_t$, for some $b'_t \in \cB$, and note that $\|1-b'_s\|\le \ep/3$. Use Lemma~\ref{lm:Popa} to find unitaries $v_1, \dots, v_m$ in $\cB$ such that $\|m^{-1} \sum_{j=1}^mv_jb'_t\alpha_{ts^{-1}}(v_j)^*\| \le \ep(3|F|)^{-1}$, for all $t \in F \setminus \{s\}$. It then follows that 
$$\big\|\frac{1}{m} \sum_{j=1}^m v_j(y_0-b'_su_s)\alpha_{s^{-1}}(v_j)^*\big\| \le \ep/3,$$
 and hence that
 $$\big\| \frac{1}{m} \sum_{j=1}^m v_j y\alpha_{s^{-1}}(v_j)^*-u_s\big\| =
 \big\|\frac{1}{m} \sum_{j=1}^m v_j(y-u_s)\alpha_{s^{-1}}(v_j)^*\big\| \le \ep.$$
 Since $\sum_{j=1}^m v_jy\alpha_{s^{-1}}(v_j)^*$ belongs to $C^*(\cB,x)$, we conclude that $u_s \in C^*(\cB,x)$. 
\hfill $\square$

\begin{example} \label{ex:On} For $2 \le n < \infty$ consider the Cuntz algebra $\cO_n$ and its sub-\Cs{} $\cB_n$ isomorphic to the UHF-algebra of type $n^\infty$, which arises as the fixed point algebra under the canonical circle action of $\cO_n$. Let $E \colon \cO_n \to \cB_n$ be the  canonical conditional expectation (obtained by integrating with respect to the circle action). Then $E$ is faithful and has the pinching property, cf.\ \cite{Cuntz:On} (use the projection $Q$ constructed in the proof of Proposition 1.7). Hence $\cB_n \subseteq \cO_n$ is $C^*$-irreducible. 

It is well-known that $\cO_n = C^*(\cB_n,s_1)$, and that the isometry $s_1$ induces a (non-unital) endomorphism $\rho$ on $\cB_n$ by $\rho(b) = s_1bs_1^*$. It this sense we can write $\cO_n$ as a crossed product $\cB_n \rtimes_\rho \N$ over the semigroup $\N$. Similar to the situation of Theorem~\ref{thm:PICS}, each proper intermediate \Cs{} of the inclusion $\cB_n \subseteq \cO_n$ is equal to $\cB_n \rtimes_\rho d\N$, for some $d \ge 2$. This claim can be proved using the same methods as in the proof of Theorem~\ref{thm:PICS} presented above (details can be found on the website of the author). 
The crossed product $\cB_n \rtimes_\rho d\N$ is equal to $C^*(\cB_n,s_1^d)$ and also to the fixed-point algebra $\cO_n^{\Z/d}$ with respect to the order $d$ automorphism on $\cO_n$ given by $s_j \mapsto \omega s_j$, where $\omega$ is a primitive $d$th root of the unit. 
\end{example}

\noindent We end this section by considering inclusions of the form $C^*_\lambda(\Gamma) \subseteq \cB \rtimes_{\mathrm{red}} \Gamma$, and we offer in the theorem below a sharpening of \cite[Theorem 1.1]{Amrutam-Kalantar-ETDS2020}.

\newpage 
\begin{theorem} \label{thm:AK}
Let $\Gamma$ be a discrete countable $C^*$-simple group acting on a unital \Cs{} $\cB$. Then the following conditions are equivalent:
\begin{enumerate}
\item $C^*_\lambda(\Gamma) \subseteq \cB \rtimes_{\mathrm{red}} \Gamma$ is $C^*$-irreducible,
\item for each non-zero positive  $a \in \cB$ there exist $t_1, \dots, t_n \in \Gamma$ such that $\sum_{j=1}^n u_{t_j}^* a u_{t_j} \ge 1_\cB$,
\item each state  $\phi$ on $\cB$ is $\Gamma$-faithful, i.e., if $\phi(u_t^*au_t)=0$, for some positive 
$a \in \cB$ and for all $t \in \Gamma$, then $a =0$,
\item there exists $\mu \in \mathrm{Prob}(\Gamma)$ such that each $\mu$-stationary state $\phi$ on $\cB$ is faithful.
\end{enumerate}
\end{theorem}

\noindent A state 
$\phi$ on $\cB$ is \emph{$\mu$-stationary} if $\mu * \phi = \phi$, where $\mu * \phi = \sum_{t \in \Gamma} \mu(t) \, t.\phi$ is the convolution of $\phi$ by $\mu$, and where $(t.\phi)(a) = \phi(u_t^*au_t)$, for $t \in \Gamma$ and $a\in \cB$.

It was shown in \cite[Theorem 1.1]{Amrutam-Kalantar-ETDS2020} that (iv) $\Rightarrow$ (i) holds with the extra assumption that $\mu$ is $C^*$-simple (i.e., the canonical trace $\tau_0$ on $C^*_\lambda(\Gamma)$ is the only $\mu$-invariant state on $C^*_\lambda(\Gamma)$). The condition in \cite[Theorem 1.1]{Amrutam-Kalantar-ETDS2020} is hence formally stronger than the condition in (iv). As shown in the proof of (iii) $\Rightarrow$ (iv) below, one can in (iv) take any $\mu \in \mathrm{Prob}(\Gamma)$ as long as the support of $\mu$ generates $\Gamma$ as a semi-group.

In the case of an abelian \Cs{} $\cB$ condition (ii) above is equivalent to the action being minimal (cf.\ \cite[Corollary 1.2]{Amrutam-Kalantar-ETDS2020}). In general, condition (ii) implies minimality of the action, and it can be viewed as a strong ``mixing property''. Recall that it was shown in 
\cite[Theorem 1.8]{BKKO} that $\cB \rtimes_{\mathrm{red}} \Gamma$ is simple whenever $\Gamma$ is $C^*$-simple and the action $\Gamma \curvearrowright \cB$ is minimal.

It was shown in \cite{Amrutam:2019} that each intermediate \Cs{} of the inclusion $C^*_\lambda(\Gamma) \subseteq \cB \rtimes_{\mathrm{red}} \Gamma$ is of the form $\cB  \rtimes_{\mathrm{red}} \Gamma$ for some $\Gamma$-invariant sub-\Cs{} of $\cB$ if $\Gamma$ has the approximation property (AP) of Haagerup and Kraus, and if the kernel of the action $\Gamma \curvearrowright \cB$ is a plump subgroup (i.e., that condition (v) of Theorem~\ref{thm:C*-group} holds).  This provides yet another example showing that $C^*$-irreducible inclusions are ``rigid''.

The assumption of countability of $\Gamma$ is only used to prove the implication (iv) $\Rightarrow$ (iii).

\begin{proof} (ii) $\Rightarrow$ (i). Take a non-zero positive element $b \in \cB \rtimes_{\mathrm{red}} \Gamma$. We must prove that $b$ is full relatively to $C^*_\lambda(\Gamma)$. Note that $E(b) \in \cB$ is positive and non-zero, where $E \colon \cB \rtimes_{\mathrm{red}} \Gamma \to \cB$ is the canonical conditional expectation. 

Consider the dense subset $\mathcal{A}_0 = \mathrm{span} \{au_t : a \in \cB, t \in \Gamma\}$ of $\cB \rtimes_{\mathrm{red}} \Gamma$.   Assuming that (ii) holds, we can find $s_1, \dots, s_m \in \Gamma$ such that $\sum_{j=1}^m u_{s_j}^* E(b) u_{s_j} \ge 1_\cB$. Set $b' = \sum_{j=1}^m u_{s_j}^* b u_{s_j}$. Then $E(b') = \sum_{j=1}^m E(u_{s_j}^* b u_{s_j})= \sum_{j=1}^m u_{s_j}^* E(b) u_{s_j} \ge 1_\cB$.
Find $b'' \in \mathcal{A}_0$ such that $E(b'') = E(b') \ge 1_\cB$ and $\|b'-b''\| < 1/3$. 

It was shown in \cite[Lemma 2.1]{Amrutam-Kalantar-ETDS2020} that 
$$\big\| \frac{1}{n} \sum_{j=1}^n u_{t_j}^* a u_r u_{t_j}\big\| \le \|a\| \big\| \frac{1}{n} \sum_{j=1}^n u_{t_j^{-1}rt_j} \big\|,$$
for each $a \in \cB$ and all $r, t_1, \dots, t_n \in \Gamma$. Combining this estimate with Theorem~\ref{thm:C*-simple} shows that  one can find elements $t_1, \dots,  t_n \in \Gamma$ such that 
$$\big\| \frac{1}{n} \sum_{j=1}^n u_{t_j}^* (b'' - E(b'')) u_{t_j}\big\| < 1/3.$$
Hence  $\frac{1}{n} \sum_{j=1}^n u_{t_j}^* b'' u_{t_j}\ge \frac23 \cdot 1_\cB$, so $\frac{1}{n} \sum_{j=1}^n u_{t_j}^* b' u_{t_j} \ge \frac13 \cdot 1_\cB$. 
This shows that $b'$, and hence $b$, are full relatively to $C^*_\lambda(\Gamma)$, so (i) holds.

(iii) $\Rightarrow$ (ii). Suppose that (ii) fails and take a positive non-zero element $a \in \cB$ witnessing the failure of (ii). Then, for each finite subset $\mathcal{F} \subseteq \Gamma$, the element $\sum_{t \in \mathcal{F}} u_t^*au_t$ is non-invertible, so the set $T(\mathcal{F})$ of all states $\phi$ on $\cB$ for which  $\phi(u_t^*au_t)= 0$, for all $t \in \mathcal{F}$, is non-empty. Hence $T = \bigcap_{\mathcal{F}} T(\mathcal{F})$ is also non-empty (the intersection is over all finite subsets $\mathcal{F}$ of $\Gamma$) and any state $\phi$ in $T$ will satisfy $\phi(u_tau_t^*)=0$, for all $t \in \Gamma$. This shows that (iii) fails.

(iii) $\Rightarrow$ (ii) is trivial, and (ii) $\Rightarrow$ (i) follows from Lemma~\ref{lm:rel-full}.

(iii) $\Rightarrow$ (iv). Let $a \in \cB$ be positive. For $\mu \in \mathrm{Prob}(\Gamma)$ and for each state $\phi$ on $\cB$ we have $(\mu * \phi)(a) = 0$ if and only if $\phi(u_t^* a u_t)=0$, for all $t \in \mathrm{supp}(\mu)$. Note also that $\mathrm{supp}(\mu^k)$ is the set of all $t \in \Gamma$ that can be written as a product of $k$ elements from $\mathrm{supp}(\mu)$. Hence, if $\phi$ is $\mu$-stationary, then $\phi(a) = 0$ if and only if $\phi(u_t^* a u_t)=0$, for all $t$ in the semi-group generated by $\mathrm{supp}(\mu)$. 

Thus, if (iii) holds, and if we take $\mu \in \mathrm{Prob}(\Gamma)$ such that $\mathrm{supp}(\mu) = \Gamma$ (or such that the semigroup generated by $\mathrm{supp}(\mu)$ is $\Gamma$), then all $\mu$-stationary states on $\cB$ are faithful.

(iv) $\Rightarrow$ (iii). Suppose that (iii) does not hold, and let $\mu \in \mathrm{Prob}(\Gamma)$. Let $a \in 
\cB$ be a non-zero positive element for which there exists a state $\phi$ on $\cB$ such that $\phi(u_t^* a u_t) =0$, for all $t \in \Gamma$. Then $(\mu^k * \phi)(a) = 0$, for all $k \ge 0$. 
Set $\phi_n = n^{-1}\sum_{k=1}^n \mu^k * \phi$, and let $\phi_0$ be  a weak$^*$ accumulation point for the sequence $\{\phi_n\}_{n=1}^\infty$. Then $\phi_0$ is $\mu$-stationary and non-faithful, since $\phi_0(a) =0$.
Hence (iv) does not hold.
\end{proof}

\noindent The example below shows that $C^*$-irreducibility and usual irreducibility (trivial relative commutant) are not equivalent properties for the class of inclusions covered by Theorem~\ref{thm:AK}. We first need an elementary lemma, whose proof is left to the reader (the given estimate is hardly best possible, but suffices for our purposes).

\begin{lemma} \label{lm:D} Let $\cA$ be a unital \Cs, let $x \in \cA$, and let $f_1, f_2, f_3 \in \cA$ be pairwise orthogonal projections. Then
\vspace{-.5cm}
$$\|x\| \le 2 \sum_{j=1}^3  \|(1-f_j)x(1-f_j)\|.$$
\end{lemma}

\begin{example} \label{ex:irr-nirr} We show here that any $C^*$-simple group $\Gamma$ admits an action on the Cuntz algera $\cO_\infty$ such that the inclusion $C^*_\lambda(\Gamma) \subseteq \cO_\infty \rtimes_{\mathrm{red}} \Gamma$ is irreducible, but not $C^*$-irreducible, while both algebras of the inclusion are simple.  

Choose an action of $\Gamma$  on $\N$ such that for each finite subset $F \subseteq \N$ there exists $t \in \Gamma$ such that $t.F \cap F = \emptyset$ (eg., take the action $\Gamma \curvearrowright \Gamma$ given by left multiplication, and identify the latter copy of $\Gamma$ with $\N$). Let $\alpha$ denote the action of $\Gamma$ on $\cO_\infty$ given by $\alpha_t(s_k) = s_{t.k}$, for $k \in \N$ and $t \in \Gamma$, where $\{s_k\}_{k=1}^\infty$ are the canonical generators of $\cO_\infty$.
Set $e_{k} = s_ks_k^*$. Then $\{e_{k}\}_{k=1}^\infty$ are pairwise orthogonal projections, and $\alpha_t(e_{k}) = e_{t.k}$. We conclude that  (ii) in Theorem~\ref{thm:AK} does not hold with $a = e_{1}$, so $C^*_\lambda(\Gamma) \subseteq \cO_\infty \rtimes_{\mathrm{red}} \Gamma$ is not $C^*$-irreducible.

Since $\Gamma$ is icc, being  $C^*$-simple, it follows that $C^*_\lambda(\Gamma)' \cap (\cO_\infty \rtimes_{\mathrm{red}} \Gamma) = \cO_\infty^\Gamma$.  We claim that $\cO_\infty^\Gamma = \C$. Take $x \in \cO_\infty^\Gamma$. Let $\cA_0$ be the  dense $^*$-algebra generated by the isometries $\{s_k\}_{k=1}^\infty$. Let $\ep >0$, and choose $y\in \cA_0$ with $\|x-y\| \le \ep$. Then there is a finite set $\mathcal{F}$ of  multi-indices  in $\bigcup_{\ell \ge 0} \N^\ell$ such that 
$$y = \alpha_0 \cdot 1 + y_0,  \quad y_0 = \sum_{(\mu,\nu) \in \mathcal{G}} \alpha_{\mu,\nu} \, s_\mu s_\nu^*, 
\qquad \alpha_0, \alpha_{\mu,\nu} \in \C,$$
where $\mathcal{G} = (\mathcal{F} \times \mathcal{F})\setminus \{(\emptyset,\emptyset)\}$. 
 Let $F \subseteq \N$ be the union of the supports of the multi-indices in $\mathcal{F}$. Set $e = \sum_{k \in F} e_{k}$. Then $s_\mu = e s_\mu$ and $s_\nu^* = s_\nu^*e$, for all  $\mu, \nu \in \mathcal{F} \setminus \{\emptyset\}$. It follows that $(1-e)(y-\alpha_0 \cdot 1)(1-e) = (1-e)y_0(1-e)=0$. 

Choose $t_1,t_2 \in \Gamma$ such that the sets $F, t_1.F, t_2.F$ are pairwise disjoint. Set $f_j = \alpha_{t_j}(e) = \sum_{k \in t_j.F} e_{k}$. Then 
$$(1-f_j)(y-\alpha_0 \cdot 1)(1-f_j) = \alpha_{t_j}\big((1-e)(\alpha_{t_j}^{-1}(y)-\alpha_0 \cdot 1-y+\alpha_0 \cdot 1)(1-e)\big),$$
so $\|(1-f_j)(y-\alpha_0 \cdot 1)(1-f_j)\| \le \| \alpha_{t_j}^{-1}(y)-y\| \le 2\ep$, for $j=1,2$. By Lemma~\ref{lm:D} we conclude that $\|y-\alpha_0 \cdot 1\| \le 8\ep$. This shows that $\mathrm{dist}(x,\C) \le 9\ep$, so $x \in \C$.
\end{example}

\section{Inductive limits} \label{sec:indlimit}

\noindent In this section we consider when a unital inclusion of \Cs{} arising from inductive limits is $C^*$-irreducible. The general set-up is as follows: Given a commutative diagram:

\begin{equation} \label{eq:indlimit}
\begin{split}
\xymatrix@R-.3pc{\cB_1 \ar[r]^-{\mu_1} \ar[d]_-{\iota_1} & \cB_2 \ar[r]^-{\mu_2} \ar[d]_-{\iota_2} & \cB_3 \ar[r]^-{\mu_3} \ar[d]_-{\iota_3}  & \cdots \ar[r] & 
\cB  \ar[d]^\iota \\ 
\cA_1 \ar[r]^{\lambda_1} & \cA_2 \ar[r]^-{\lambda_2} & \cA_3 \ar[r]^-{\lambda_3} & \cdots \ar[r] & \cA
}
\end{split}
\end{equation}
where $\cB$ and $\cA$ are the inductive limit \Cs s of the sequence of \Cs s in the first, respectively, the second row. We will assume that all maps in the diagram are injective. Let $\mu_{m,n} \colon \cB_n \to \cB_m$ and $\mu_{\infty,n} \colon \cB_n \to \cB$ denote the (composed) connecting maps, and likewise for $\lambda_{m,n} \colon \cA_n \to \cA_m$ and $\lambda_{\infty,n} \colon \cA_n \to \cA$.

We say that the diagram is \emph{regular} if $\iota_n(\cB_n) = \cA_n \cap \lambda_{\infty,n}^{-1}(\iota(\cB))$, for all $n \ge 1$. If all maps in the diagram above are the inclusion mappings, then this condition reads: $\cB_n = \cA_n \cap \cB$. Using this notation we always have $\cB_n \subseteq \cA_n \cap \cB$, and we can make any diagram as above regular by replacing $\cB_n$ with $\cA_n \cap \cB$. If one just assumes that $\cA_n \cap \cB_{n+1}= \cB_n$, for all $n \ge 1$, then it follows that $\cA_n \cap \Big(\bigcup_{k=1}^\infty \cB_k\Big) = \cB_n$, for all $n \ge 1$.

For an example of a non-regular diagram take any inductive limit representing $\cA$ with non-surjective connecting mappings $\lambda_n$, and set
$\cB_1 = \C$, $\cB_n = \cA_{n-1}$, $\iota_n = \lambda_{n-1}$, and $\mu_n = \lambda_{n-1}$, for $n \ge 2$. Then $\iota$ is surjective (so the resulting inclusion is trivial), but the inclusions  $\iota_n \colon \cB_n \to \cA_n$ are non-trivial.

Suppose that we are given $\cA$ as the inductive limit as in \eqref{eq:indlimit} and a unital inclusion $\iota \colon \cB \to \cA$. Then $\cB$ arises as in \eqref{eq:indlimit} if and only if $\iota(\cB) = \overline{\bigcup_{n=1}^\infty \iota(\cB) \cap \lambda_{\infty,n}(\cA_n)}$, in which case one can take $\cB_n =  \lambda_{\infty,n}^{-1}(\iota(\cB)) \subseteq \cA_n$, for $n \ge 1$, which, in addition, will make the diagram \eqref{eq:indlimit}  regular.

\begin{lemma} \label{lm:A0}
Let $\cB \subseteq \cA$ be a unital inclusion of \Cs s and let $\cA_0$ be a dense sub-$^*$-algebra of $\cA$ closed under continuous function calculus. Then $\cB \subseteq \cA$ is $C^*$-irreducible if  each non-zero positive element in $\cA_0$ is full relatively to $\cB$.
\end{lemma}

\begin{proof}  Let $a$ be a non-zero positive element in $\cA$ and find a positive element $a_0 \in \cA_0$ with $\delta:= \|a-a_0\| < \|a_0\|$. Choose $\delta < \ep < \|a_0\|$. As in the proof of Lemma~\ref{lm:C*-irr} we find that $(a_0-\ep)_+ \le \varphi(a_0)a\varphi(a_0)$. By assumption, and because $(a_0-\ep)_+$ is a non-zero positive element in $\cA_0$, there exist $x_1, \dots, x_n \in \cB$ such that $\sum_{j=1}^n x_j^*(a_0-\ep)_+x_j \ge 1_\cA$. It follows that $\sum_{j=1}^n y_j^*ay_j \ge 1_\cA$, when $y_j = \varphi(a_0)x_j \in \cB$.
\end{proof}

\noindent It is a well-known and heavily used fact that one can construct simple \Cs s as the inductive limit $\cA= \varinjlim \cA_n$ of possibly non-simple \Cs s $\cA_n$. One just has to make sure that each non-zero element in any of the \Cs s $\cA_n$ eventually becomes full in $\cA_m$, for some $m \ge n$. A similar result holds for constructing $C^*$-irreducible inclusions:

\begin{proposition} \label{prop:indlimit} 
Suppose we are given a system as in \eqref{eq:indlimit}.  Then $\iota \colon \cB \to \cA$ is $C^*$-irreducible if and only if for each $n \ge 1$ and each non-zero positive element $a \in \cA_n$ there exists $m \ge n$ such that $\lambda_{m,n}(a) \in \cA_m$ is full relatively to $\iota_m(\cB_m)$.

In particular, if each inclusion $\iota_n \colon \cB_n \to \cA_n$ is $C^*$-irreducible, then so is $\iota \colon \cB \to \cA$.
\end{proposition}

\begin{proof} We show that each non-zero positive element $a_0$ in $\cA_0 := \bigcup_{n=1}^\infty \lambda_{\infty,n}(\cA_n) \subseteq \cA$ is full in $\cB$, cf.\ Lemma~\ref{lm:A0}. Write $a_0=\lambda_{\infty,n}(a)$ for some $n \ge 1$ and some non-zero positive element $a \in \cA_n$. By assumption we may find $m \ge n$ and $x_1, \dots, x_k \in \cB_m$ with $\sum_{j=1}^k \iota_m(x_j)^* \lambda_{m,n}(a) \iota_m(x_j) \ge 1_{\cA_m}$. Set $y_j = \mu_{\infty,m}(x_j) \in \cB$. Then
$$\sum_{j=1}^k \iota(y_j)^* a_0 \, \iota(y_j) = \lambda_{\infty,m}\Big( \sum_{j=1}^k \iota_m(x_j)^* \lambda_{m,n}(a) \iota_m(x_j) \Big) \ge 1_\cA,$$
as desired.
\end{proof}

\begin{proposition} \label{prop:commuting-square} Given a diagram as in \eqref{eq:indlimit}, and suppose that for each $n \ge 1$ there is a conditional expectation $E_{\cA_n,\cB_n}$ of $\cA_n$ onto $\cB_n$ making each square
$$ \xymatrix{\cA_n \ar[r]^-{\lambda_n} \ar[d]_-{E_{\cA_n,\cB_n}} & \cA_{n+1} \ar[d]^-{E_{\cA_{n+1},\cB_{n+1}}}  \\ \cB_n \ar[r]^-{\mu_n} & \cB_{n+1}}
$$
commutative. Then \eqref{eq:indlimit} is regular and there is a conditional expectation $E \colon \cA \to \cB$ commuting with each $E_{\cA_n,\cB_n}$, for $n \ge 1$.
\end{proposition}

\begin{proof} For ease of notation suppose that all maps $\mu_n$, $\lambda_n$ and $\iota_n$ are inclusion mappings. By the commutativity assumption, the conditional expectations $E_{\cA_n,\cB_n}$ extend to a contractive positive linear map $E_0 \colon \bigcup_{n=1}^\infty \cA_n \to \bigcup_{n=1}^\infty \cB_n$, which again, by continuity, extends to a conditional expectation $E \colon \cA \to \cB$ commuting with the $E_{\cA_n,\cB_n}$. If $a \in \cA_n \cap \cB$, then $a = E(a) = E_{\cA_n,\cB_n}(a) \in \cB_n$, so \eqref{eq:indlimit} is regular.
\end{proof}

\noindent If $\cB \subseteq \cA$ is an inclusion of finite dimensional \Cs s, then $\cB' \cap \cA = \C$ implies $\cB = \cA$, so the inclusion can never be $C^*$-irreducible, unless it is trivial. One can still construct $C^*$-irreducible inclusions of AF-algebras using the more general setting of Proposition~\ref{prop:indlimit}, that simplifies a bit further in the case of inductive limits of finite dimensional \Cs s.

\begin{lemma} \label{lm:perp}
Let $\cB \subseteq \cA$ be a unital inclusion of finite dimensional \Cs s $\cB$ and $\cA$. Let $E \colon \cA \to \cA \cap \cB'$ be the conditional expectation defined by
$$E(x) = \int_{\cU(\cB)} uxu^* \, d\mu(u), \qquad x \in \cA,$$
where $\mu$ is the Haar measure on $\cU(\cB)$.  The following conditions are equivalent for each positive element $a \in \cA$:
\begin{enumerate}
\item $a$ is full relatively to $\cB$,
\item $E(a)$ is invertible,
\item no non-zero projection $p \in \cA \cap \cB'$ is orthogonal to $a$.
\end{enumerate}
\end{lemma}

\begin{proof} Note first that $E$ defined in the lemma indeed is a conditional expectation onto $\cA \cap \cB'$.

(ii) $\Rightarrow$ (i): Approximate $E(a)$ by finite Riemann sums to obtain an invertible positive element in $\mathrm{conv}\{uau^* : u \in \cU(\cB)\}$ that witnesses that $a$ is full relatively to $\cB$. 

(i) $\Rightarrow$ (iii): Choose elements $x_1, \dots, x_n \in \cB$ such that $\sum_{j=1}^n x_j^*ax_j \ge 1_\cA$. Let $p \in \cA \cap \cB'$ be non-zero. Then $0 \ne p \le \sum_{j=1}^n x_j^* pap x_j$, so $p$ is not orthogonal to $a$. 

(iii) $\Rightarrow$ (ii): To show that $E(a)$ is invertible it suffices to show that $pE(a)p \ne 0$ for all  non-zero projections $p \in \cA \cap \cB'$. By the definition of $E$ we have $pE(a)p = E(pap)$. Since $E$ is faithful and $pap \ne 0$ by (iii) we conclude that $pE(a)p \ne 0$.
\end{proof}

\noindent One can replace projections in Lemma~\ref{lm:perp} (iii) with positive elements.

Say that two sub-\Cs s $\cB_1$ and $\cB_2$ of a common \Cs{} $\cA$ are \emph{everywhere non-orthogonal} if no pair of non-zero positive elements $b_1 \in \cB_1$ and $b_2 \in \cB_2$ are orthogonal to each other. For example, the \Cs s $\cB_1 \otimes 1$ and $1 \otimes \cB_2$ are everywhere non-orthogonal in $\cB_1 \otimes \cB_2$ for any pair of (non-zero) \Cs s $\cB_1$ and $\cB_2$. On the other hand, $\cB$ is everywhere non-orthogonal to itself if and only if $\cB=\C$ (since otherwise $\cB$ contains pairwise orthogonal non-zero positive elements).

\begin{corollary} \label{cor:indlimit}
Suppose we are given a system as in \eqref{eq:indlimit} and that each \Cs{} $\cA_n$ and $\cB_n$ are finite dimensional.  Then $\iota \colon \cB \to \cA$ is $C^*$-irreducible if the algebras $\lambda_n(\cA_n) \subseteq \cA_{n+1}$ and $\cA_{n+1} \cap \iota_{n+1}(\cB_{n+1})' \subseteq \cA_{n+1}$ are everywhere non-orthogonal, for all $n \ge 1$.  
\end{corollary}

\begin{proof} This follows immediately from Proposition~\ref{prop:indlimit} and Lemma~\ref{lm:perp}.
\end{proof}

\begin{lemma} \label{lm:irr-embedding} Let $k \ge d \ge 2$ be integers. Then there exists a unitary $u$ in $M_d \otimes M_k$ such that the \Cs s $M_d \otimes 1_k$ and $u^*(M_d \otimes 1_k)u$ are everywhere non-orthogonal in $M_d \otimes M_k$.
\end{lemma}

\begin{proof} Choose a projection $f \in M_k$ of dimension $d$ and choose a unital \sh{} $\rho \colon M_d \to fM_k f$. The two (unital) \sh s $M_d \to M_d \otimes M_k$ given by $x \mapsto x \otimes 1_k$ and $x \mapsto 1_d \otimes \rho(x) + x \otimes (1_k-f)$ are unitarily equivalent, and hence there exists a unitary $u \in M_d \otimes M_k$ such that $u^*(x \otimes 1_k)u = 1_d \otimes \rho(x) + x \otimes (1_k-f)$, for all $x \in M_d$. Since $(x \otimes 1_k)\big(1_d \otimes \rho(y) + y \otimes (1_k-f)\big) = x \otimes \rho(y) + xy \otimes (1_k-f)$ is non-zero for all non-zero $x,y \in M_d$, it follows that $M_d \otimes 1_k$ and $u^*(M_d \otimes 1_k)u$ are everywhere non-orthogonal.
\end{proof}

\begin{remark} In the situation of Lemma~\ref{lm:irr-embedding}, if $d,k \ge 2$ are integers and there exists a unitary $u$ in $M_d \otimes M_k$ such that the \Cs s $M_d \otimes 1_k$ and $u^*(M_d \otimes 1_k)u$ are everywhere non-orthogonal, then $k^2 \ge d+1$, as we shall show below.

For each unit vector $x \in \C^d$, let $p_x$ be the $1$-dimensional projection onto $\C x$.  If $u \in M_d \otimes M_k$ is as claimed, then $u^*(p_x \otimes 1_k)u(p_y \otimes 1_k) \ne 0$, and hence $(p_x \otimes 1_k)u(p_y \otimes 1_k) \ne 0$, for all unit vectors $x,y \in \C^d$. Let $\{\rho_j\}_{j=1}^{k^2}$ be a basis for $M_k^*$ and set $u_j = (\mathrm{id}_d \otimes \rho_j)(u) \in M_d$. Then $(p_x \otimes 1_k)u(p_y \otimes 1_k) \ne 0$ if and only if there exists $j$ such that 
$$p_x u_j p_y = (\mathrm{id}_d \otimes \rho_j)\big((p_x \otimes 1_k)u(p_y \otimes 1_k)\big) \ne 0,$$ which in turns is equivalent to $\langle u_j y, x\rangle \ne 0$. If this holds for all non-zero $x \in \C^d$, then  $u_j$, $1 \le j \le k^2$, must satisfy
 $$\mathrm{span} \{ u_j y : 1  \le j \le k^2\} = \C^d$$
 for all unit vectors $y \in \C^d$. This clearly implies that  $k^2 \ge d$. At closer inspection we can improve this estimate to $k^2 \ge d+1$, since for any $j \ne j'$ there exists a unit vector $y \in \C^d$ such that $u_j y$ and $u_{j'} y$ are proportional.
 
The smallest number $k$ for which the conclusion of Lemma~\ref{lm:irr-embedding} holds, for a given $d \ge 2$, must therefore satisfy $\sqrt{d+1} \le k \le d$. I leave it as a curious problem to narrow down this interval, or to derive an exact formula for $k$. Relatedly, the smallest number $m$ of matrices $A_1, \dots, A_{m} \in M_d$ for which $\mathrm{span}\{A_j \, y : 1 \le j \le m\} = \C^d$, for all unit vectors $y \in \C^d$, satisfies $d+1 \le m \le k^2$ by the argument above. 
\end{remark}

\noindent Theorem \ref{thm:UHF} below is an application of  Corollary~\ref{cor:indlimit} and Lemma~\ref{lm:irr-embedding}. The theorem and its proof provide a recipe for how to construct $C^*$-irreducible inclusions of UHF-algebras, and, we expect, also of simple AF-algebras.

\begin{theorem} \label{thm:UHF} Let $\cA$ and $\cB$ be UHF-algebras so that $\cB$ admits a unital embedding into $\cA$.  Then there exists a $C^*$-irreducible embedding $\iota \colon \cB \to \cA$.
\end{theorem}

\begin{proof} Choose sequences $\{k_n\}_{n=1}^\infty$ and $\{\ell_n\}_{n=1}^\infty$ of integers $\ge 2$, such that $k_n$ is a proper divisor in $\ell_n$, for all $n \ge 1$, and $\cA = \bigotimes_{n=1}^\infty M_{\ell_n}$ and $\cB = \bigotimes_{n=1}^\infty M_{k_n}$. Write $\ell_n = k_n d_n$. Upon replacing $k_{n+1}$ with $k_{n+1} k_{n+2} \cdots k_{n+m}$, for some $m \ge 1$, and likewise for $\ell_{n+1}$ --- which does not change $\cB$ or $\cA$ --- we may assume that $k_{n+1} \ge d_1 d_2 \cdots d_n$. 

We construct a commuting diagram:
\begin{equation} \label{eq:UHF}
\begin{split}
\xymatrix{M_{k_1} \ar[d]_-{\iota_1} \ar[r] & M_{k_1} \otimes M_{k_2} \ar[d]_-{\iota_2} \ar[r] 
& M_{k_1} \otimes M_{k_2} \otimes M_{k_3} \ar[d]_-{\iota_3} \ar[r] & \cdots \ar[r] & \cB \ar[d]^-\iota
\\
M_{\ell_1} \ar[r] & M_{\ell_1} \otimes M_{\ell_2} \ar[r] & M_{\ell_1} \otimes M_{\ell_2} \otimes M_{\ell_3} \ar[r] & \cdots \ar[r] & \cA}
\end{split}
\end{equation}
where the horisontal maps are the canonical ones $x \mapsto x \otimes 1_{k_{n+1}}$, respectively, $x \mapsto x \otimes 1_{\ell_{n+1}}$, and where the vertical maps $\iota_n$ will be defined to be unital \sh s such that the two algebras
\begin{equation} \label{eq:A}
\iota_{n+1}(M_{k_1} \otimes \cdots \otimes M_{k_{n+1}})' \cap M_{\ell_1} \otimes \cdots \otimes M_{\ell_{n+1}} \quad \text{and} \quad M_{\ell_1} \otimes \cdots \otimes M_{\ell_n} \otimes 1_{\ell_{n+1}}
\end{equation}
are everywhere non-orthogonal.
This will ensure that $\iota \colon \cB \to \cA$ is $C^*$-irreducible, cf.\ Corollary~\ref{cor:indlimit}.

Choose $\iota_1$ to be any unital \sh. Suppose that $n \ge 1$ and that $\iota_n$ has been constructed. We proceed to construct $\iota_{n+1}$. Set $k = k_1 k_2 \cdots k_n$, $\ell = \ell_1 \ell_2 \cdots \ell_n$, and $d = d_1 d_2 \cdots d_n$, so that $\ell = d k$. Identify $M_{k_1} \otimes \cdots \otimes M_{k_n}$ with $M_k$ and $M_{\ell_1} \otimes \cdots \otimes M_{\ell_n}$ with $M_k \otimes M_d$ in such a way that $\iota_n(x) = x \otimes 1_\ell$. Identify $M_{\ell_{n+1}}$ with $M_{k_{n+1}} \otimes M_{d_{n+1}}$. In this notation, the $n$th square in the diagram \eqref{eq:UHF} above takes the form:
$$\xymatrix@C+3cm{M_k \ar[d]_-{\iota_n = \mathrm{id}_k \otimes 1_d}  \ar[r]^-{\mathrm{id}_k \otimes 1_{k_{n+1}}} & M_k \otimes M_{k_{n+1} }\ar@{-->}[d]^-{\iota_{n+1} = \mathrm{id}_k \otimes j} \\ M_k \otimes M_d 
\ar[r]_-{\mathrm{id}_{k} \otimes \mathrm{id}_d \otimes 1_{k_{n+1}} \otimes 1_{d_{n+1}} } & M_k \otimes M_d \otimes M_{k_{n+1}} \otimes M_{d_{n+1}}}$$
where the \sh{} $j \colon M_{k_{n+1}} \to M_d \otimes M_{k_{n+1}} \otimes M_{d_{n+1}}$ is defined as follows ensuring that Equation~\eqref{eq:A} holds. Choose a unitary $u \in M_d \otimes M_{k_{n+1}}$ such that $M_d \otimes 1_{k_{n+1}}$ and $u^*(M_d \otimes 1_{k_{n+1}})u$ are everywhere orthogonal. Set $j(x) = u^*(1_d \otimes x)u \otimes 1_{d_{n+1}}$, for $x \in M_{k_{n+1}}$. 
Then 
$$\iota_{n+1}(M_k \otimes M_{k_{n+1}}) = M_k \otimes u^*(1_d \otimes M_{k_{n+1}})u \otimes 1_{d_{n+1}}.$$ 
The commutant of this algebra is $1_k \otimes u^*(M_d \otimes 1_{k_{n+1}})u \otimes M_{d_{n+1}}$, which by the choice of $u$ is everywhere non-orthogonal to $M_k \otimes M_d \otimes 1_{k_{n+1}} \otimes 1_{d_{n+1}} = M_{\ell_1} \otimes \cdots \otimes M_{\ell_n} \otimes 1_{\ell_{n+1}}$, as desired.
\end{proof}

\begin{remark} It seems plausible that the conclusion of Theorem~\ref{thm:UHF} holds more generally for any pair of infinite dimensional simple unital AF-algebras $\cA$ and $\cB$ whenever the latter admits a unital embedding into the former.
\end{remark}

\begin{remark} While AF-algebras are completely classified by their ordered $K_0$-group, the question of classifying inclusions $\cB \subseteq \cA$ of UHF-algebras (or AF-algebras) is far more subtle (even under the extra assumptions such as $C^*$-irreducibility). An inclusion $\cB \to \cA$ between AF-algebras induces a map $K_0(\cB) \to K_0(\cA)$ which classifies the inclusion up to \emph{approximate unitary equivalence}. To understand the inclusion (up to conjugacy) we will need a classification of the inclusion map up to \emph{unitary equivalence}. 
\end{remark}

\noindent It is well-known that sub-\Cs s of an AF-algebra need not be AF. Blackadar, \cite{Bla:symmetries},  constructed an example of a $\Z/2$-action $\alpha$ of the CAR-algebra $\cA$ so that the fixed point algebra $\cA^\alpha$ is not AF (in fact, $\cA^\alpha$ is a Bunce-Deddens algebra). The inclusion $\cA^\alpha \subseteq \cA$ is $C^*$-irreducible and of index 2. This example still leaves open the following question (that probably has a negative answer):

\begin{question} Let $\cB \subseteq \cA$ be a $C^*$-irreducible inclusion of simple AF-algebras. Is it true that every intermediate \Cs{} is also AF?\footnote{This question has subsequently  been answered in the negative in \cite{Echterhoff-Rordam-2021}.}
\end{question}

\begin{example} Here is another---more conceptual---example of $C^*$-irreducible inclusions of UHF-algebras (that admits many generalizations): Let $\cB = \bigotimes_{n=1}^\infty M_{k_n}(\C)$ be a UHF-algebra with respect to some sequence $\{k_n\}$ of integers $\ge 2$. Let $d \ge 2$ be an integer and choose an outer action of $\Z_d$ on $\cB$ with the property that it leaves invariant each finite dimensional sub-\Cs{} $\bigotimes_{n=1}^N M_{k_n}(\C)$ of $\cB$. One can for example take $\alpha = \bigotimes_{n=1}^\infty \mathrm{Ad}_{u_n}$, where $u_n$ is a unitary in $M_{k_n}(\C)$ of order $d$, and such that $u_n^k$ is non-scalar when $k = 1,2, \dots, d-1$. It then follows from Theorem~\ref{thm:crossed-product} that $\cB \subseteq \cB \rtimes \Z_d$ is $C^*$-irreducible. Moreover, 
$$\cB \rtimes \Z_d = \varinjlim \Big( \bigotimes_{n=1}^N M_{k_n}(\C) \rtimes \Z_d\Big),$$
and $\bigotimes_{n=1}^N M_{k_n}(\C) \rtimes \Z_d$ is isomorphic to the direct sum of $d$ copies of $\bigotimes_{n=1}^N M_{k_n}(\C)$, so $\cB \rtimes \Z_d$ is an AF-algebra.
\end{example}

\noindent We conclude this section by showing that any unital inclusion $\cB \subseteq \cA$ of AF-algebras (simple or not) admits an inductive limit decomposition as in \eqref{eq:indlimit}. For sub-\Cs s $\cA, \cB$ of a common \Cs{} $\mathcal{E}$, write $\cB \subseteq^\delta \cA$ if $\mathrm{dist}(b,\cA) \le \delta$, for all $b$ in the unit ball of $\cB$.

We shall use the following powerful theorem of Christensen:

\begin{theorem}[{Christensen, \cite{Christensen:Acta}}] \label{thm:EC}
Let $\mathcal{E}$ be a unital \Cs, let $\cA,\cB$ be unital sub-\Cs s of $\mathcal{E}$, with $\cB$ finite dimensional, let $0 < \delta < 10^{-4}$, and assume that $\cB \subseteq^\delta \cA$. Then there exist a unitary element $u \in \mathcal{E}$ such that $u\cB u^* \subseteq \cA$ and $\|1-u\| \le 120 \, \delta^{1/2}$.
\end{theorem}

\noindent The lemma below is standard AF-algebra knowledge. It is an easy consequence of Christensen's theorem above, but can be proved directly with less heavy machinery using stable relations of matrix units. 

\begin{lemma} \label{lm:AF} Given a unital inclusion $\cB \subseteq \mathcal{E}$, where $\cB$ is finite dimensional and $\mathcal{E}$ is an AF-algebra, and given a finite subset $\mathcal{F} \subseteq \mathcal{E}$, and  $\ep >0$. Then there exists a unital finite dimensional intermediate \Cs{} $\cB \subseteq \cA \subseteq \mathcal{E}$ such that and $\mathrm{dist}(f,\cA) < \ep$, for all $f \in \mathcal{F}$.
\end{lemma}

\noindent We shall need Christensen's theorem in the following adaption:

\begin{lemma} \label{lm:AF2}
Given an integer $n \ge 1$ and $\ep > 0$. Then there exists $\delta>0$ such that whenever $\mathcal{E}$ is a unital AF-algebra with unital finite dimensional sub-\Cs s 
$$\cB_1 \subseteq \cB_2 \subseteq \cdots \subseteq \cB_n \subseteq \cA_j, \qquad j=1,2,$$ 
satisfying $\cA_1 \subseteq^\delta \cA_2$, then there exist a unitary $u \in \mathcal{E}$ with $\|1-u\| \le \ep$, $u\cA_1 u^* \subseteq \cA_2$, and $u\cB_j u^* = \cB_j$, for $1 \le j \le n$. 
\end{lemma}

\begin{proof} We prove this by repeated applications of Christensen's theorem as follows.
For $\delta >0$ (to be determined below), set $\gamma_{n+1}= 120 \, \delta^{1/2}$ and define $\eta_j>0$ and $\gamma_j>0$, $1 \le j \le n$, inductively as follows: $\eta_j = 120 \, (2\gamma_{j+1})^{1/2}$ and $\gamma_j = \gamma_{j+1}+\eta_j$. Choose $0 <\delta < 10^{-4}$ such that $\gamma_1 \le \ep$ and such that $2\gamma_j < 10^{-4}$, for $2 \le j \le n$. 

At step one use Christensen's theorem to find a unitary $u_{0} \in \mathcal{E}$ with $\|1-u_{0}\| \le \gamma_{n+1}$ and $u_0\cA_1 u_0^* \subseteq \cA_2$.  Then $\cB_n \subseteq^{2\gamma_{n+1}} u_0\cB_n u_0^*$. Note that $\cB_n$ and $u_0\cB_n u_0^*$ both are sub-\Cs s of $\cA_2$. Hence there exists a unitary $v_n \in \cA_2$ with $\|1-v_n\| \le \eta_n$ and $v_n\cB_n v_n^* = u_0 \cB_n u_0^*$. Set $u_n = v_n^*u_0$. It follows that $u_n \cB_n u_n^* = \cB_n$, $u_n \cA_1 u_n^* \subseteq \cA_2$, and $\|1-u_n\| \le \gamma_n$. 

For the next step, note that $\cB_{n-1} \subseteq^{2\gamma_n} u_n\cB_{n-1}u_n^*$ and $\cB_{n-1}, u_n\cB_{n-1}u_n^*$ both are sub-\Cs s of $\cB_n$. Hence there exists a unitary $v_{n-1}$ in $\cB_n$ with $v_{n-1}\cB_{n-1} v_{n-1}^* = u_n\cB_{n-1}u_n^*$ and $\|1-v_{n-1}\| \le \eta_{n-1}$. The unitary element $u_{n-1} = v_{n-1}^*u_n$ then satisfies $u_{n-1}\cB_j u_{n-1}^* = \cB_j$, for $j=n-1,n$, $u_{n-1}\cA_1u_{n-1}^* \subseteq \cA_2$, and $\|1-u_{n-1}\| \le \gamma_{n-1}$. 

Continue like this until we have arrived at a unitary $u_2 \in \mathcal{E}$ satisfying $u_2 \cB_j u_2^* = \cB_j$, for $2  \le j \le n$, $u_2 \cA_1 u_2^* \subseteq \cA_2$, and $\|1-u_2\| \le \gamma_2$. Then $\cB_{1} \subseteq^{2\gamma_2} u_2\cB_{1}u_2^*$ and $\cB_{1}, u_2\cB_{1}u_2^*$ are both sub-\Cs s of $\cB_2$. Hence there exists a unitary $v_{1}$ in $\cB_2$ with $v_{1}\cB_{1} v_{1}^* = u_2\cB_{1}u_2^*$ and $\|1-v_{1}\| \le \eta_{1}$. The unitary element $u:=u_{1} = v_{1}^*u_2$ then has the desired properties. 
\end{proof}

\begin{proposition} Let $\cB \subseteq \cA$ be a unital inclusion of AF-algebras. Then there exist  an increasing sequence $\{\cA_n\}_{n=1}^\infty$ of finite dimensional sub-\Cs s of $\cA$ such that 
$$\cA = \overline{\bigcup_{n=1}^\infty \cA_n}, \qquad \text{and} \qquad \cB = \overline{\bigcup_{n=1}^\infty (\cA_n \cap \cB)}.$$
In particular, with $\cB_n = \cA_n \cap \cB$, the inclusion $\cB \subseteq \cA$ arises from a diagram as in \eqref{eq:indlimit} with all $\cB_n$ and $\cA_n$ finite dimensional.
\end{proposition}

\begin{proof} Let $\{\cB_n\}_{n=1}^\infty$ be an increasing sequence of finite dimensional sub-\Cs s of $\cB$ whose union $\bigcup_{n=1}^\infty \cB_n$ is dense in $\cB$, and choose a dense sequence  
$\{a_n\}_{n=1}^\infty$ in $\cA$. 
For each $n \ge 1$ we find a finite dimensional \Cs{} $\widetilde{\cA}_n$ and a unitary $u_n \in \cA$ such that
\begin{enumerate}
\item[(a)] $u_n \widetilde{\cA}_n u_n^* \subseteq \widetilde{\cA}_{n+1}$,
\item[(b)] $\cB_n \subseteq \widetilde{\cA}_n$,
\item[(c)] $u_n \cB_j u_n^* = \cB_j$, for $1 \le j \le n$,
\item[(d)] $\|1-u_n\| \le 2^{-n}$,
\item[(e)] $\mathrm{dist}(a_j,\widetilde{\cA}_n) < 1/n$, for $1 \le j \le n$.
\end{enumerate}

To start, it follows from Lemma~\ref{lm:AF} that we can find a finite dimensional sub-\Cs{} $\widetilde{\cA}_1$ of $\cA$ satisfying (b) and (e). 

Suppose that $n \ge 1$ and that we have found $\widetilde{\cA}_1, \widetilde{\cA}_2, \dots, \widetilde{\cA}_n$ and unitaries $u_1, \dots, u_{n-1}$ as above. (If $n=1$, then no unitary $u_j$ has been found yet.) Choose $\delta >0$ as in Lemma~\ref{lm:AF2} corresponding to our given $n \ge 1$ and to $\ep = 2^{-n}$.  Use Lemma~\ref{lm:AF} to find a finite dimensional sub-\Cs{} $\widetilde{\cA}_{n+1}$ of $\cA$ satisfying (b), (e) and $\widetilde{\cA}_n \subseteq^\delta \widetilde{\cA}_{n+1}$, and then use Lemma~\ref{lm:AF2} to find a unitary $u_n \in \cA$ satisfying (a), (c) and (d).

Set $v_n = \lim_{N \to \infty} u_N u_{N-1} \cdots u_n$ (the sequence converges by (d)). Then $v_n$ is a unitary element in $\cA$ and $\|v_n-1\| \le 2^{-n+1}$. Set $\cA_n = v_n \widetilde{\cA}_n v_n^*$. Then $\cB_n \subseteq \cA_n$ by (b) and (c). As  $v_{n+1} = v_nu_n$ we see from (a) that $\cA_n \subseteq \cA_{n+1}$, for all $n \ge 1$. Since
$$\mathrm{dist}(a_j,\cA_n) \le \mathrm{dist}(a_j,\widetilde{\cA}_n) + \|a_j - v_na_jv_n^*\| \le 1/n+2\|v_n-1\| \|a_j\|,$$
for $1 \le j \le n$,  we conclude that $\bigcup_{n=1}^\infty \cA_n$ is dense in $\cA$. Finally, $\bigcup_{n=1}^\infty (\cA_n \cap \cB)$ is dense in $\cB$ because $\cB_n \subseteq \cA_n \cap \cB$.  This completes the proof.
\end{proof}

\section{Tensor products} \label{sec:tensor}

\noindent We investigate here how the property of being $C^*$-irreducible behaves under forming tensor products. Using the theorem of Zacharias and Zsido (Theorem~\ref{thm:tensorI}) we already mentioned (and proved) the following:

\begin{theorem}  \label{thm:tensorII} 
Let $\cB \subseteq \cA$ be a $C^*$-irreducible inclusion and let $\mathcal{E}$ be a unital simple \Cs{} with Wassermann's property (S). Then $\mathcal{E} \otimes \cB  \subseteq \mathcal{E} \otimes \cA$ is $C^*$-irreducible.
\end{theorem}

\noindent Recall that all nuclear \Cs s have property (S). The Zacharias--Zsido theorem further says that the map $\cD \mapsto \mathcal{E} \otimes \cD$ gives a bijection between the set of intermediate \Cs s of the inclusion $\cB \subseteq \cA$  and the set of intermediate \Cs s of the inclusion $\mathcal{E} \otimes \cB  \subseteq \mathcal{E} \otimes \cA$.

It follows in particular that whenever you have a $C^*$-irreducible inclusion,  one can arrange to make it $\cZ$-stable, $\cO_\infty$-stable, $\cO_2$-stable etc upon tensoring by $\cZ$, $\cO_\infty$, and $\cO_2$, respectively.

\begin{proposition} \label{prop:tensorIII} 
Let $\cB_i \subseteq \cA_i$, $i \in I$, be a (finite or infinite) family of unital inclusions of \Cs s. Set $\cB = \bigotimes_{i \in I} \cB_i$ and $\cA= \bigotimes_{i \in I} \cA_i$.
\begin{enumerate}
\item If each $\cB_i \subseteq \cA_i$ has the relative Dixmier property with respect to some faithful tracial state $\tau_i$ on $\cA_i$, then $\cB \subseteq \cA$ also has the relative Dixmier property with respect to the faithful tracial state $\tau=\otimes_{i \in I} \, \tau_i$. In particular, the inclusion $\cB \subseteq \cA$ is $C^*$-irreducible.
\item If each $\cB_i \subseteq \cA_i$ has the relative excision property with respect to a faithful state $\psi_i$ on $\cA_i$, for each $i \in I$, then $\cB \subseteq \cA$ has the excision property relatively to the faithful state $\psi = \otimes_{i \in I} \, \psi_i$ on $\cA$. In particular, if each $\cB_i$ is simple, then the inclusion $\cB \subseteq \cA$ is $C^*$-irreducible.
\end{enumerate}
\end{proposition}

\begin{proof} Denote by $\boxdot_{i \in I} \cB_i$ and  $\boxdot_{i \in I} \cA_i$ the set of elementary tensors of the form $\otimes_{i \in I} c_i$ with $c_i \in \cB_i$, respectively, $c_i \in \cA_i$, and with $c_i \ne 1$ for only finitely many $i \in I$. Note that $\cA$ is the closure of the span of $\boxdot_{i \in I} \cA_i$, and similarly for $\cB$.

 (i). For each $i \in I$, set $\mathcal{E}_i = \mathrm{conv} \{ {\mathrm{Ad}}_u : u \in \cU(\cB_i)\}$ and let $\mathcal{E}$ be the set of self-maps on $\cA$ of the form $E = \otimes_{i \in I} E_i$, where $E_i \in \mathcal{E}_i$ and $E_i \ne \mathrm{id}_{\cA_i}$ for at most finitely many $i \in I$. Let $X$ be the set of elements $a \in \cA$ for which the relative Dixmier property holds in the following sense: for each $\ep >0$ there exists $E \in \mathcal{E}$ such that $\|E(a)-\tau(a) \cdot 1\| < \ep$. 
 
 The assumption that each inclusion $\cB_i \subseteq \cA_i$ has the relative Dixmier property implies that $\boxdot_{i \in I} \cA_i \subseteq X$. Secondly, since each $E \in \mathcal{E}$ maps $\boxdot_{i \in I} \cA_i$ into itself, we conclude that also the span of $\boxdot_{i \in I} \cA_i$ is contained in $X$. Finally, as $X$ is closed, it follows that $X = \cA$.
 
 (ii). By assumption we can, for each $i \in I$, find nets $\{h(\alpha,i)\}_\alpha$ (indexed over the same upwards directed set) of positive elements in $\cB_i$ with $\|h(\alpha,i)\|=1$ such that 
 $$\lim_\alpha \|h(\alpha,i)^{1/2}ah(\alpha,i)^{1/2}-\psi_i(a) \, h(\alpha,i)\| = 0, \qquad a \in \cA_i.$$
 For each finite subset $I_0$ of $I$ and each $\alpha$ set
 $h(\alpha,I_0) = \otimes_{i \in I} \, \tilde{h}(\alpha,i)$, where $\tilde{h}(\alpha,i)= h(\alpha,i)$, when $i \in I_0$ and where $\tilde{h}(\alpha,i)=1_{\cB_i}$, when $i \notin I_0$. Then $\{h(\alpha,I_0)\}_{(\alpha, I_0)}$ is a net of positive elements in $\cB$ with $\|h(\alpha,I_0)\|=1$ and
 $$\lim_{(\alpha,I_0)} \|h(\alpha,I_0)^{1/2}ah(\alpha,I_0)^{1/2}-\psi(a) \, h(\alpha,I_0)\| = 0, \qquad a \in \cA,$$
 thus proving that $\psi$ can be excised relatively to $\cB$.
 
 The claims about $C^*$-irreducibility in both (i) and (ii) follows from Proposition~\ref{prop:M}.
\end{proof}

\noindent The lattice of intermediate sub-\Cs s  between $\bigotimes_{i \in I} \cB_i \subseteq \bigotimes_{i \in I} \cA_i$ contains all \Cs s of the form $\bigotimes_{i \in I} {\mathcal{D}_i}$, with $\cB_i \subseteq \mathcal{D}_i \subseteq \cA_i$. However, not all intermediate sub-\Cs s are of this form. If $I = \{1,2\}$ and $\cB_i \subset \cA_i$ are strict unital inclusions of \Cs s, then, for example, $C^*(\cB_1 \otimes \cA_2, \cA_1 \otimes \cB_2)$ is a proper intermediate \Cs, which is not a tensor product of intermediate \Cs s. A possible generalization of the tensor splitting theorem to this situation could be that each intermediate \Cs{} of the inclusion $\bigotimes_{i \in I} \cB_i \subseteq \bigotimes_{i \in I} \cA_i$ is generated by \Cs s of the form $\bigotimes_{i \in I} {\mathcal{D}_i}$, with $\cB_i \subseteq \mathcal{D}_i \subseteq \cA_i$.

\begin{question}\label{q:tensor}
 Let $\cB_i \subseteq \cA_i$, $i \in I$, be a (finite or infinite) family of $C^*$-irreducible unital inclusions of \Cs s. Is it true that $\bigotimes_{i \in I} \cB_i \subseteq \bigotimes_{i \in I} \cA_i$ is also $C^*$-irreducible?
\end{question}

\noindent Proposition~\ref{prop:tensorIII} answers this question in the affirmative under somewhat stronger conditions on the inclusions $\cB_i \subseteq \cA_i$. In particular, if $\cN_i \subseteq \cM_i$, $i \in I$,  are (proper) inclusions of II$_1$-factors with $[\cM_i : \cN_i] < \infty$, for all $i \in I$, then the inclusion $\bigotimes_{i \in I} \cN_i \subseteq \bigotimes_{i \in I} \cM_i$ is  $C^*$-irreducible, cf.\ Theorem~\ref{thm:vNirr}, 
 (where the tensor products are the minimal $C^*$-tensor product). If $I$ is finite, then the inclusion 
${\overline{\bigotimes}}_{i \in I} \cN_i \subseteq {\overline{\bigotimes}}_{i \in I} \cM_i$ 
of II$_1$-factors is also $C^*$-irreducible, while this fails when $I$ is infinite, because the inclusion has infinite Jones index.

Note also that it suffices to answer Question~\ref{q:tensor} when $|I| = 2$. An easy induction argument will namely then settle the cases where the index set is finite, and the general case will finally follow from Proposition~\ref{prop:indlimit}.

 \newpage
{\small{
\bibliographystyle{amsplain}

\providecommand{\bysame}{\leavevmode\hbox to3em{\hrulefill}\thinspace}
\providecommand{\MR}{\relax\ifhmode\unskip\space\fi MR }
\providecommand{\MRhref}[2]{%
  \href{http://www.ams.org/mathscinet-getitem?mr=#1}{#2}
}
\providecommand{\href}[2]{#2}


%
%
}}

\vspace{1cm}

\noindent
Mikael R\o rdam \\
Department of Mathematical Sciences \\
University of Copenhagen \\ 
Universitetsparken 5, DK-2100, Copenhagen \O \\
Denmark \\
Email: rordam@math.ku.dk\\
WWW: http://web.math.ku.dk/$\sim$rordam/
%
%
%
%
%
%
%

\end{document}